\documentclass[secthm,seceqn,amsthm,ussrhead,reqno,12pt]{amsart}
\usepackage[utf8]{inputenc}
\usepackage[english]{babel}
\usepackage[symbol]{footmisc}
\usepackage{amssymb,amsmath,amsthm,amsfonts,xcolor,enumerate,hyperref, comment,longtable, cleveref}
\usepackage{verbatim}
\usepackage{times}
\usepackage{cite}
\usepackage{pdflscape}
\usepackage{ulem}
\usepackage[mathcal]{euscript}

\usepackage{hyperref}
\usepackage{cancel}
\usepackage{stmaryrd}
\usepackage{tasks}
 
\usepackage{amsfonts}
\usepackage{amssymb}
\usepackage{times}
\usepackage{xcolor} 
\usepackage{url}
\usepackage{float}
\usepackage{tasks}

\usepackage{cite}
\usepackage{hyperref}

\usepackage{amsfonts}
\usepackage{amsmath}
\usepackage{eurosym}
\usepackage{geometry}

\usepackage{caption,booktabs}

\theoremstyle{plain}
\newtheorem{theorem}{Theorem}
\newtheorem{lemma}[theorem]{Lemma}
\newtheorem{proposition}[theorem]{Proposition}
\newtheorem{remark}[theorem]{Remark}
\newtheorem{notation}[theorem]{Notation}
\newtheorem{example}[theorem]{Example}
\newtheorem{definition}[theorem]{Definition}
\newtheorem{corollary}[theorem]{Corollary}
\begin{document}

\noindent{\Large Non-degenerate evolution algebras}\footnote{
The first part of this work is supported by the PCI of the UCA `Teor\'\i a de Lie y Teor\'\i a de Espacios de Banach', by the PAI with project number
 FQM298, by  the 2014-2020 ERDF Operational Programme and by the Department of Economy, Knowledge, Business and University of the Regional Government of Andalusia. Project reference: FEDER-UCA18-107643; and by the Spanish project `Algebras no conmutativas y de caminos de Leavitt. Algebras de evoluci\'on. Estructuras de Lie y variedades
de Einstein'; 
by the Spanish Government through the Ministry of Universities grant `Margarita Salas', funded by the European Union - NextGenerationEU;
FCT   UIDB/00324/2020,  UIDB/MAT/00212/2020, UIDP/MAT/00212/2020.
The second part of this work is supported by the Russian Science Foundation under grant 22-71-10001.}
  
 \bigskip

\begin{center}

 {\bf
 Antonio Jesús Calderón Martín\footnote{University of Cadiz, Puerto Real, Spain; \ ajesus.calderon@uca.es},
Amir Fernández Ouaridi\footnote{Department of Mathematics, University of Cádiz, Puerto Real, Espa\~na; 
 CMUC, Department of Mathematics, University of Coimbra, Coimbra, Portugal; \ amir.fernandez.ouaridi@gmail.com}
\&   
Ivan Kaygorodov
\footnote{CMA-UBI, Universidade da Beira Interior, Covilh\~{a}, Portugal;\  
Moscow Center for Fundamental and Applied Mathematics, Moscow,   Russia; \
Saint Petersburg  University, Russia; \ kaygorodov.ivan@gmail.com}

}

\end{center}

 \bigskip

 \bigskip  

\noindent{\bf Abstract}:
{\it In this paper we introduce a new  invariant for a non-degenerate  evolution algebra, which consists of an ordered  sequence   of evolution algebras of lower dimension,  belonging all of them to a specific family. We use this invariant to propose a method to classify non-degenerate evolution algebras, and we apply it up to dimension 3. We also use it to describe the derivations of some families of evolution algebras and the variety of evolution algebras with square  not greater than 1.}

 \bigskip

 \bigskip


 \bigskip  
 

 \bigskip

\tableofcontents

\newpage
\section*{Introduction}

It seems that the notion of evolution algebras was first introduced in 2004 in the thesis of Tian (see also, his first paper about elementary algebraic properties of  evolution algebras published together  
with Vojtěchovský  in 2006 \cite{TIAN-vo} and a book published in 2008  \cite{TIAN}).
Evolution algebras are a new type of genetic algebras that describe the functioning of the non-mendelian genetic, including, for example, the asexual propagation or the asexual inheritance, as well as physical phenomena, both dynamic and kinematic. In geometry, evolution algebras can describe the motions of particles in a graph embedded in a $3$-manifold. Also, some stochastic processes, such as Markov chains on countable state spaces, can be modelled using evolution algebras. 
Relations of evolution algebras with other structures  and the history of studies of them in the first decade are given in a recently published  survey  \cite{survey}.

In recent years, many different aspects of the evolution algebras have been considered by many authors \cite{NYM,candido,velasco,yolder,yolannal,CSV2,EL21,survey,yolt,FFN, LR13,LR17, M14, CLR,RM,LLR,P21,
yol22,csv16, CSV2t, cgo13, CLTV,EL2,CGMGM,sri,TIAN,TIAN-vo}. 
There are works in which its structure is studied \cite{csv16, CSV2, yolt}, a description of the evolution algebras up to dimension three \cite{CSV2t}, a classification of the simple evolution algebras up to dimension three \cite{yol22} and a classification of the nilpotent evolution algebras up to dimension five \cite{EL2}. Also, there are studies of some types of derivations of these algebras \cite{cgo13}, different representations using graphs \cite{CSV2t, EL2}, and even applications to certain specific population problems, among many others.

The main purpose of this paper is to introduce an invariant for non-degenerate evolution algebras and to show how it can be used to construct a general method to classify non-degenerate evolution algebras and also, to provide an example of the application of this method for dimension two and three. 
It should be mentioned that all known (particular) classifications of $3$-dimensional evolution algebras are given without defining non-isomorphic algebras \cite{CSV2t,3dimCV}. 
The present classification gives a classification of non-isomorphic $3$-dimensional non-degenerate evolution algebras.
Additionally, we show some relations between this invariant and the derivations of the algebra, and, in particular, we describe the derivations of some families of evolution algebras. Our starting point is the decomposition of a finite-dimensional  evolution algebra as a direct sum of adequate linear subspaces given in \cite{NYM}. 

\section{Definitions and preliminary results}

 Unlike other classes of non-associative algebras, the class of evolution algebras is not defined by a set of identities, instead, they are defined in the following way:

\begin{definition}
Let $\mathbb{F}$ be a field. An {\it evolution algebra} is an algebra ${\mathbb A}$ provided with
a basis ${\bf B}=\{e_{i} \ \vert \ i\in \Lambda \}$ such that $e_{i}e_{j}=0$
whenever $i\neq j$.
Such a basis ${\bf B}$ is called a {\it natural basis}.
Fixed a natural basis ${\bf B}$ in ${\mathbb A},$ the scalars $w_{ki}\in \mathbb{F}$ such that
 $e_{i}^{2} =\sum_{k\in \Lambda} w _{ki}e_{k}$ will be called the {\it structure constants} of ${\mathbb A}${\it \ relative to} ${\bf B}$ and the matrix ${\bf M}_{\bf B}({\mathbb A}):= (w_{k i})$ is said to be the {\it structure matrix of} ${\mathbb A}$ {\it relative to} ${\bf B}$. Thus, every evolution algebra is uniquely determined by its structure matrix.
\end{definition}

Any algebra in this paper will be finite-dimensional and over an algebraically closed base field $\mathbb{F}$. We will denote by $\mathbb  S_n$ the symmetric group of all permutations of the set $\{1,\ldots,n\}$. We  continue with the introduction of some definitions and results.

\begin{definition}
A {\it natural vector} in an evolution algebra ${\mathbb A}$ is a vector that can be extended into a natural basis. A {\it natural set} is a set of linearly independent vectors that can be extended into a natural basis of ${\mathbb A}$.
Any (linear) subspace ${\mathbb E}$ of ${\mathbb A}$  generated by a natural set will be called an {\it extending evolution subspace}  of ${\mathbb A}$. Such a family will be called an {\it extending natural basis} of ${\mathbb E}$.
\end{definition}


\begin{definition}
Given an algebra ${\mathbb A}$, the annihilator of the algebra is  the ideal
$${\rm Ann}({\mathbb A})=\{ x\in{\mathbb A}: x{\mathbb A}+{\mathbb A}x=0  \}$$
We say an algebra is non-degenerate if it has a trivial annihilator.
\end{definition}

The following remark is a well-known fact about non-degenerate evolution algebras.

\begin{remark}

Given a non-degenerate evolution algebra, then for every natural vector $e$ we have  $e^2\neq0$. Therefore, the structure matrix on any natural basis does not contain zero columns.

\end{remark}


\begin{theorem}[\cite{CSV2t}]
\label{conb}
Let ${\mathbb A}$ be an evolution algebra and let ${\bf B}=\left\{e_1,\ldots, e_n\right\}$ be a natural basis of ${\mathbb A}$ with structure matrix ${\bf M}_{\bf B}({\mathbb A})$. Then for another natural basis ${\bf B}'=\left\{f_{1},\ldots, f_{n}\right\}$ of ${\mathbb A}$ and ${\bf C} =(c_{ij})$ change of basis between the natural bases ${\bf B}'$ and ${\bf B}$, i.e., $f_j=\sum_i c_{ij}e_i$, then
$${\bf M}_{\bf B'}({\mathbb A})={\bf C}^{-1}{\bf M}_{\bf B}({\mathbb A}){\bf C}^{(2)},$$

where ${\bf C}^{(2)}=(c_{ij}^2)$.
\end{theorem}








Our classifying method parts from the decomposition below.

\begin{theorem}[\cite{NYM}, Theorem 2.11]
\label{nym}
Let ${\mathbb A}$ be an evolution algebra. Then
$${\mathbb A}= {\rm Ann}({\mathbb A})\oplus  {\mathbb E}_1 \oplus \ldots  \oplus {\mathbb E}_r,$$
where  ${\mathbb E}_1, \ldots, {\mathbb E}_r$ are extending evolution subspaces of ${\mathbb A}$ satisfying
\begin{center}
${\rm dim}({\mathbb E}_{i}^2)=1$, \ ${\mathbb E}_{i}  {\mathbb E}_{j} = 0$ \ and \    ${\rm dim}( {\mathbb E}_{i}^2 +{\mathbb E}_{j}^2)=2$ for $i \neq j$.
\end{center}Moreover, if ${\mathbb A}$ is non-degenerate, the decomposition is unique (up to the permutation of the linear subspaces).
\end{theorem}

\begin{definition}
Let ${\mathbb A}$ be a non-degenerate  evolution algebra. The unique, (up to permutation of the linear subspaces), decomposition
$${\mathbb A}=   {\mathbb E}_1 \oplus \ldots  \oplus {\mathbb E}_r,$$
with  ${\mathbb E}_1, \ldots, {\mathbb E}_r$  extending evolution subspaces of ${\mathbb A}$ satisfying conditions in Theorem \ref{nym}
will be called the standard extending evolution subspaces decomposition of ${\mathbb A}$.

\end{definition}






\subsection{The $\delta$-index  of a non-degenerate evolution algebra}

In this subsection, we are going to define the first  invariant that will be used in our classification method. It is a consequence of the following lemma.

\begin{lemma}
\label{princi}
Let ${\mathbb A}$, ${\mathbb A}'$ be two  non-degenerate evolution algebras with standard extending evolution decompositions  ${\mathbb A}= {\mathbb E}_1 \oplus \ldots  \oplus {\mathbb E}_r$ and ${\mathbb A}'= \mathbb E_1' \oplus \ldots  \oplus \mathbb E_s'$ respectively. If ${\mathbb A} \cong {\mathbb A}'$, then $r=s$ and $\phi({\mathbb E}_{i})={\mathbb E}_{\sigma(i)}'$ for a permutation $\sigma \in \mathbb S_r$.
\end{lemma}
\begin{proof}
Denote by  $\phi:{\mathbb A}\rightarrow {\mathbb A}'$  an isomorphism. Then $\phi({\mathbb E}_{i})^2=\phi({\mathbb E}_{i}^2)$ has dimension $1$. Also, $\phi({\mathbb E}_{i})^2+\phi({\mathbb E}_{j})^2=\phi({\mathbb E}_{i}^2+{\mathbb E}_{j}^2)$ for $i\neq j$ has dimension $2$.
Moreover, since $\phi({\mathbb E}_{i})\phi({\mathbb E}_{j})=\phi({\mathbb E}_{i} {\mathbb E}_{j})=0$ for $i\neq j$ then
${\mathbb A}'=\phi({\mathbb E}_1) \oplus \ldots  \oplus \phi({\mathbb E}_r)$ is a standard extending evolution decomposition of ${\mathbb A}'$. By the uniqueness of this decomposition  we get $r=s$ and $\phi({\mathbb E}_{i}) = {\mathbb E}_{\sigma(i)}'$ for some permutation $\sigma \in \mathbb S_r$.
\end{proof}

By Theorem \ref{nym}, the following index is well-defined.

\begin{definition}
Let ${\mathbb A}$ be a non-degenerate evolution algebra with standard extending subspaces decomposition ${\mathbb A}= {\mathbb E}_1 \oplus \ldots  \oplus {\mathbb E}_r$. Then we define the $\delta$-index of  ${\mathbb A}$ as the sequence
$$\delta({\mathbb A}):= \big({\rm dim}  ({\mathbb E}_{1}), {\rm dim}({\mathbb E}_{2}), \ldots, {\rm dim}({\mathbb E}_r)\big)$$
with ${\rm dim}({\mathbb E}_{i}) \geq {\rm dim}({\mathbb E}_{i+1})$.
\end{definition}

\begin{corollary}
\label{cor2}
Let ${\mathbb A}$, ${\mathbb A}'$ be two isomorphic non-degenerate evolution algebras. Then $\delta({\mathbb A})=\delta({\mathbb A}')$.
\end{corollary}

%

The converse of Corollary \ref{cor2} does not hold as the next example shows.

\begin{example}\label{exampler}
 Consider the non-degenerate evolution algebras defined by 
 \begin{center}${\mathbb A}: \ e_1^2 = e_1, \  e_2^2 = e_2$\  and \ ${\mathbb A}': \ e_1^2 = e_2,\ e_2^2 = e_1$.\end{center} 
 We have that $\delta({\mathbb A}) = \delta({\mathbb A}')=(2,1)$, but ${\mathbb A}$ is not isomorphic to ${\mathbb A}'$.
\end{example}

Note that every perfect evolution algebra has $\delta$-index equal to $(1, \ldots, 1)$. Recall that every simple   algebra is perfect. 
Regarding the changes of natural bases, we can observe the following remark.

\begin{remark}
\label{baseshape}
Given a non-degenerate evolution algebra ${\mathbb A}$ with $\delta({\mathbb A}) = (\delta_1, \ldots, \delta_r)$, and two natural bases  ${\bf B}={\bf B_1} \cup {\bf B_2} \cup \ldots  \cup {\bf B_r}$ and
${\bf B'}={\bf B'_1} \cup {\bf B'_2} \cup \ldots  \cup {\bf B'_r}$ of ${\mathbb A}$ given by the standard extending subspaces decomposition of ${\mathbb A}$.  Then the  change of  basis matrix has the following block form:
$$ \begin{pmatrix}
{\bf D}_{1} & 0 & 0 & 0  \\
0 &  {\bf D}_{2} & 0 & 0 \\
0 & 0 & \ddots & 0  \\
0 & 0 & 0   & {\bf D}_{r} \\
\end{pmatrix} \cdot {\bf P}$$
where any ${\bf D}_{i}$ is a  $\delta_i\times\delta_i$ block of natural vectors of ${\mathbb A}$ (as columns) and ${\bf P}$ is a permutation matrix.
\end{remark}

Observe that if in the above remark, we have that $\delta_i = 1$ for every $1\leq i\leq r$, then the only changes of natural basis are products of diagonal matrices by permutations, which is indeed a group with the composition that has been previously studied in \cite{CSV2t}.


\begin{theorem}\label{teo100}
\label{renorm}
Let $\mathbb A$ be a non-degenerate evolution algebra with  $\delta({\mathbb A}) = (\delta_1, \ldots, \delta_r)$. Then there is a natural basis ${\bf B}=\left\{e_{1},\ldots, e_{n}\right\}$ of ${\mathbb A}$ such that the structure matrix
${\bf M}_{\bf B}({\mathbb A})$ has the columns form $${\bf M}_{\bf B}({\mathbb A})=(v_1\ldots  v_1 | v_2 \ldots  v_2 | \ldots  \ldots  | v_r \ldots  v_r),$$
in such a way that ${\rm dim}({\mathbb{F}}v_i+{\mathbb{F}}v_j) = 2$ for $i\neq j$.
\end{theorem}

\begin{proof}

Let  ${\mathbb A}= {\mathbb E}_1 \oplus \ldots  \oplus {\mathbb E}_r$ be the standard extending subspaces decomposition of
${\mathbb A}$  and suppose ${\rm dim}({\mathbb E}_{i}) \geq {\rm dim}({\mathbb E}_{i+1})$. For any $i \in \{1,\ldots,r\}$ consider ${\bf B}_i=\left\{ e_1^i, \ldots, e_{\delta_i}^{i}\right\}$  an extending natural basis of ${\mathbb E}_i$. Then ${\bf B}=\cup_{i \in \{1,\ldots,r\}} {\bf B}_i$ is a natural basis of $\mathbb A$. For this basis we have the following structure matrix of ${\mathbb A}$ by columns:
$${\bf M}_{\bf B}({\mathbb A}) =(v_1 \,\,\, \lambda_{1,{2}}v_1\ldots  \lambda_{1,{\delta_1}} v_1 | v_2 \,\,\,
\lambda_{2,{2}}v_2\ldots  \lambda_{2,{\delta_2}} v_2 | \ldots  \ldots  | v_r\,\,\, \lambda_{r,{2}}v_r\ldots  \lambda_{r,{\delta_r}} v_r)$$
with any $\lambda_{r,s}\in {\mathbb{F}}$ and  such that
${\rm dim}({\mathbb{F}}v_i+ {\mathbb{F}}v_j) = 2$ for $i\neq j$.

Now, consider a new set ${\bf B}'$ containing the following vectors:
\begin{itemize}
    \item $f_1^i = e_1^i$ for every $1 \leq i\leq r$.

    \item $f_j^i =  \lambda_{i,j}^{-\frac{1}{2}} e_j^i$ for every $1 \leq i\leq r$ and $2 \leq j\leq \delta_{i}$.
\end{itemize}

From, Remark \ref{baseshape} the set ${\bf B}'$ is a natural basis of ${\mathbb A}$, and from Theorem \ref{conb} we have $${\bf M}_{\bf B'}({\mathbb A})=(v_1\ldots  v_1 | v_2 \ldots  v_2 | \ldots  \ldots  | v_r \ldots  v_r),$$
which completes the proof.
\end{proof}

By Theorem \ref{renorm}, fixed a basis ${\bf B}$, every non-degenerate evolution algebra ${\mathbb A}$ with  $\delta({\mathbb A}) = (\delta_1, \ldots, \delta_r)$ is isomorphic to an algebra ${\mathbb A}'$ with structure matrix  columns
$${\bf M}_{\bf B}({\mathbb A}')=(v_1\ldots  v_1 | v_2 \ldots  v_2 | \ldots  \ldots  | v_r \ldots  v_r),$$
satisfying  that ${\rm dim}({\mathbb{F}}v_i+{\mathbb{F}}v_j) = 2$ for $i\neq j$. Therefore, to obtain a classification of  non-degenerate evolution algebras, it will be enough to consider those with a structure matrix of the above form. Moreover, the changes of natural basis that preserve this shape are those with ${\bf D}_i^{\rm T} {\bf D}_{i} = \lambda_i  {\bf{1}}_{\delta_i}$ from Remark \ref{baseshape}, for some $\lambda_i\in {\mathbb{F}}$ and  where ${\bf{1}}_{\delta_i}$ denotes the ${\delta_i}\times {\delta_i}$ identity matrix.

\subsection{The $\Delta$-trace of a non-degenerate evolution algebra}
Let ${\mathbb A}$ be a non-degenerate evolution algebra with
standard extending subspaces decomposition  ${\mathbb A}= {\mathbb E}_1 \oplus \ldots  \oplus {\mathbb E}_r.$  Our next aim is to introduce a sequence of $r$ evolution algebras associated to this decomposition.
For any $i\in \{1,\ldots,r\}$,
 denote by $$\pi_{{\mathbb E}_{i}}:{\mathbb A} \to {\mathbb E}_{i}$$ the projection map onto ${\mathbb E}_{i}$. Then we  consider the algebra ${\mathbb A}_i$ defined over the vector space ${\mathbb E}_{i}$ by the product
$${\mathbb A}_i: x \cdot_{{\mathbb A}_i} y =\pi_{{\mathbb E}_{i}}(x \cdot_{\mathbb A} y)$$
for any $x, y \in {\mathbb E}_{i}$. We call to ${\mathbb A}_i$ the {\it algebra associated to} ${\mathbb E}_i$.

\begin{remark}
\label{rem12}
Let ${\mathbb A}$ be a non-degenerate evolution algebra with
standard extending subspaces decomposition ${\mathbb A}= {\mathbb E}_1 \oplus \ldots  \oplus {\mathbb E}_r.$ Then, for any $i\in \{1,\ldots,r\}$,  the algebra ${\mathbb A}_i$ associated to ${\mathbb E}_i$ satisfies:

\begin{enumerate}
    \item ${\mathbb A}_i$ is an  evolution algebra.
    \item  ${\rm dim} \ {\mathbb A}_i^2 \leq 1.$  
    \item The  structure matrix of ${\mathbb A}_i$ in any natural basis has proportional columns (see Lemma \ref{princi}).
    \item ${\mathbb A}_i$ is either the zero product  algebra ${\bf O_{n_i}}$ or  a non-degenerate evolution algebra with ${\rm dim}({\mathbb A}_{i}^{2})=1$.
\end{enumerate}

\end{remark}

\begin{definition}
Let ${\mathbb A}$ be a non-degenerate evolution algebra with
standard extending subspaces decomposition ${\mathbb A}= {\mathbb E}_1 \oplus \ldots  \oplus {\mathbb E}_r, $ and consider for any ${\mathbb E}_i$ its evolution associated algebra ${\mathbb A}_i$.
We call to the sequence of algebras $$({\mathbb A}_1,{\mathbb A}_2, \ldots,{\mathbb A}_r),$$ (unique up to permutation of the algebras), the standard sequence of algebras of $ {\mathbb A}.$
\end{definition}

\begin{example}\label{ex1}

Let ${\mathbb A}$ be an evolution algebra and let ${\bf B}=\left\{e_1, \ldots, e_5\right\}$ be a natural basis of ${\mathbb A}$ such that the structure matrix of ${\mathbb A}$ relative to the basis ${\bf B}$ is:
$${\bf M}_{\bf B}({\mathbb A})=\begin{pmatrix}
1 & 0 & -1 & 0 & 0 \\
1 & 2 & -1 & 0 & 1 \\
2 & 4 & -2 & 0 & 2 \\
2 & 0 & -2 & 0 & 0 \\
0 & 0 & 0  & 1 & 0 \\
\end{pmatrix}.$$
The unique decomposition of ${\mathbb A}$ into extending evolution subspaces is ${\mathbb E}_{1} \oplus {\mathbb E}_{2} \oplus {\mathbb E}_{3}$, where ${\mathbb E}_{1}=\langle e_1, e_3 \rangle$, ${\mathbb E}_{2}=\langle e_2, e_5 \rangle$ and ${\mathbb E}_{3}=\langle e_4 \rangle$. Therefore, $\delta({\mathbb A})=(2, 2, 1)$. Moreover, the standard sequence of algebras $({\mathbb A}_1, {\mathbb A}_2, {\mathbb A}_3)$ of $ {\mathbb A}$ is  if formed by the  evolution algebras with structure matrices:
\begin{align*}
{\bf M}_{{\bf B}_{\mathbb E_1}}({\mathbb A}_{1}) =\begin{pmatrix}
1 & -1  \\
2 & -2  \\
\end{pmatrix}, \qquad
{\bf M}_{{\bf B}_{\mathbb E_2}}({\mathbb A}_{2})=\begin{pmatrix}
2 & 1 \\
0 & 0 \\
\end{pmatrix}, \qquad
{\bf M}_{{\bf B}_{\mathbb E_3}}({\mathbb A}_{3})=\begin{pmatrix} 0 \\
\end{pmatrix},
\end{align*}
where 
${{\bf B}_{\mathbb E_1}} = \left\{e_1, e_3\right\}$, 
${{\bf B}_{\mathbb E_2}}= \left\{e_2, e_5\right\}$ and 
${{\bf B}_{\mathbb E_3}} = \left\{e_4\right\}$.
\end{example}

\medskip

By Remark \ref{rem12}, we are interested in knowing    non-degenerate evolution algebras  ${\mathbb A}$ such that ${\rm dim}({\mathbb A}^2)=1$. For any natural basis ${\bf B}$ of such an algebra ${\mathbb A}$, the structure matrix of ${\mathbb A}$ has the columns form
$${\bf M}_{\bf B}({\mathbb A})=    (v \quad \lambda_2 v \quad \ldots  \quad \lambda_{n} v),$$
where $v $ is a non-zero column and any $\lambda_r \in {\mathbb{F}}$,  also  being   $\delta({\mathbb A}) = (n)$.

The classification of this subclass of algebras is given in the following theorem:

\begin{theorem}
\label{clasifd1}
Let ${\mathbb A}$ be an $n$-dimensional non-degenerate evolution algebra with ${\rm dim}({\mathbb A}^2)=1$. We have
\begin{enumerate}
    \item if $n=1$ then ${\mathbb A}$ is isomorphic to ${\bf E}_1: e_{1}^2=e_{1}$.
    \item if $n\geq 2$ then:
    \begin{enumerate}
        \item  if $({\mathbb A}^2)^2\neq 0$, then ${\mathbb A}$ is isomorphic to     ${\bf E}_n: e_{1}^2=e_{2}^2=\ldots =e_{n}^2=e_1$.
        \item  if $({\mathbb A}^2)^2=0$, then ${\mathbb A}$ is isomorphic to ${\bf I}_n: e_{1}^2=e_{2}^2=\ldots =e_{n}^2=e_1 +  {\bf i}  e_2$.
    \end{enumerate}
\end{enumerate}
\end{theorem}
\begin{proof}
Let ${\mathbb A}$ be a non-degenerate evolution algebra with  ${\rm dim}({\mathbb A}^2)=1$. The case $n=1$ is trivial, therefore we will suppose  $n\geq 2$. Since ${\mathbb{F}}$ is algebraically closed, there is a natural basis of ${\mathbb A}$, ${\bf B}=\left\{e_1, \ldots, e_n\right\}$, such that $e_i^2=v=v_1e_1+\ldots +v_n e_n$ by Theorem \ref{renorm}.

\begin{enumerate}[{\it (a)}]
    \item Suppose $({\mathbb A}^2)^2\neq 0$. Now, if $\lambda=\sum \lambda_i e_i$ is an idempotent, then $\lambda^2=\sum \lambda_i^2 e_i^2=\sum \lambda_i^2 e_1^2=\sum \lambda_i e_i$. Matching the scalars we have the following system of equations:
\begin{longtable}{lcllclc lcl}
  $(\sum \lambda_i^2)v_1$&$=$&$\lambda_1,$ & 
  $(\sum \lambda_i^2)v_2$&$ =$&$ \lambda_2,$ &$
  \ldots,$ &
  $(\sum \lambda_i^2)v_n$&$=$&$\lambda_n,$
\end{longtable}
which has solution if and only if $v^2\neq0$, and the solution is $\lambda_i=\frac{v_i}{\sum v_i^2}$. Since $({\mathbb A}^2)^2\neq 0$ implies $v^2\neq0$, we have found a natural idempotent $\lambda$. From here, constructing the change of basis is straightforward. 

    \item Suppose $({\mathbb A}^2)^2= 0$, then $v^2=0$. Let $u\in {\mathbb A}$ be such that $uv=v$ and define a vector $w:=-\frac{1}{2}u^2 + u$. Clearly $wv=v$ and $w^2=0$.
    Since  ${\rm dim}\big({\rm ker}(\sum v_i x_i)\big)=n-1$, there exist a set $\left\{\lambda_3, \ldots, \lambda_n\right\}$ of linearly independent orthogonal vectors such that $\lambda_i v=0$. Assume without loss of generality that $\lambda_i^2=v$. Define $z_i:= \lambda_i-\lambda_i w$, then $\left\{z_3, \ldots, z_n\right\}$ is a set of linearly independent 
    orthogonal vectors such that $z_i^2=v$ and $z_i v=z_i w=0$.
    Finally, construct a linear map $\phi: {\mathbb A} \to {\bf I}_n$ such that
\begin{longtable}{lcl}
  $\phi(v)$&$=$&$e_1+{\bf i} e_2,$\\
  $\phi(w)$&$=$&$\frac{1}{2}e_1-\frac{{\bf i}}{2} e_2,$\\
  $\phi(z_i) $&$ =$&$e_i, \textrm{ for } i=3, \ldots, n.$
\end{longtable}
It is easy to verify that this is the isomorphism that we are looking for.
\end{enumerate}
\end{proof}


Theorem \ref{clasifd1} allows us to introduce  an order relation in the class of   non-degenerate evolution algebras with  ${\rm dim}({\mathbb A}^2)\leq 1$ as follows.

\begin{definition}

Let ${\mathbb A}$ and $ {\mathbb A}'$ be two non-degenerate evolution algebras such that   
${\rm dim}({\mathbb A}^2)\leq 1,$ 
${\rm dim}(({\mathbb A'})^2)\leq 1$. We will say  that
 ${\mathbb A} \geq_{*} {\mathbb A'}$  if:

\begin{itemize}
\item Either ${\rm dim}({\mathbb A}) \geq {\rm dim}({\mathbb A}')$ or

\item ${\rm dim}({\mathbb A}) = {\rm dim}({\mathbb A}')$ and
\begin{itemize}
    \item[(i)]  ${\mathbb A}$ is isomorphic to ${\bf O}_{n}$ and  ${\mathbb A}'$ is isomorphic to ${\bf O}_{n}$  or to ${\bf E}_{n}$ or to ${\bf I}_{n}$.

    \item[(ii)]  ${\mathbb A}$ is isomorphic to $ {\bf E}_{n}$ and ${\mathbb A}'$ is isomorphic to ${\bf E}_{n}$  or to ${\bf I}_{n}$.

    \item[(iii)]  ${\mathbb A}$ is isomorphic to ${\bf I}_{n}$ and ${\mathbb A}'$ is also isomorphic to ${\bf I}_{n}$.
\end{itemize}
\end{itemize}

\end{definition}

Theorem \ref{clasifd1} also allows us to introduce the next notation:
\begin{notation}

Let $\mathbb A$ be an $n$-dimensional non-degenerate evolution algebra such that ${\rm dim}({\mathbb A}^2)\leq 1$. We denote by
$$\overline{\mathbb A} =
\begin{cases}
      {\bf O}_n & \textrm{if } {\mathbb A}\textrm{ is isomorphic to }  {\bf O}_n \\
      {\bf E}_n & \textrm{if } {\mathbb A}\textrm{ is isomorphic to } {\bf E}_n \\
      {\bf I}_n & \textrm{if } {\mathbb A}\textrm{ is isomorphic to } {\bf I}_n
   \end{cases}
$$

\end{notation}


Let us now introduce the main invariant in our development.

\begin{definition}
Let ${\mathbb A}$ be a non-degenerate evolution algebra
with  standard sequence of algebras $({\mathbb A}_1, {\mathbb A}_2, \ldots,{\mathbb A}_r)$ in such a way that  ${\mathbb A}_i\geq_{*} {\mathbb A}_{i+1}$. Then we define the $\Delta$-trace of ${\mathbb A}$ as
$$\Delta({\mathbb A}) := (\overline{{\mathbb A}_1},\ldots, \overline{{\mathbb A}_r}).$$
\end{definition}

\begin{example}

Let ${\mathbb A}$ be the  evolution algebra in Example \ref{ex1}. Its  standard sequence of algebras was $({\mathbb A}_1, {\mathbb A}_2, {\mathbb A}_3)$ with
\begin{align*}
{\bf M}_{{\bf B}_{\mathbb E_1}}({\mathbb A}_{1}) =\begin{pmatrix}
1 & -1  \\
2 & -2  \\
\end{pmatrix}, \qquad
{\bf M}_{{\bf B}_{\mathbb E_2}}({\mathbb A}_{2})=\begin{pmatrix}
2 & 1 \\
0 & 0 \\
\end{pmatrix}, \qquad
{\bf M}_{{\bf B}_{\mathbb E_3}}({\mathbb A}_{3})=\begin{pmatrix} 0 \\
\end{pmatrix}.
\end{align*}

Since  ${\mathbb A}_1$ and ${\mathbb A}_2$ are both  isomorphic to ${\bf E}_2$ and ${\mathbb A}_3$ is isomorphic to ${\bf O_1}$, we have that
$$\Delta({\mathbb A}) := ({\bf E}_2, {\bf E}_2,{\bf O_1} ).$$

\end{example}

\medskip

This notion generalizes in some sense the notion of the diagonal subspace of an evolution algebra with a unique natural basis (for example, see \cite{yol22}) into the context of non-degenerate evolution algebras. Note that evolution algebras with a unique natural basis are precisely the non-degenerate evolution algebras with $\delta$-index a sequence of ones, sometimes called 2LI (see \cite{NYM}).  
The following result shows that $\Delta$ is  an invariant up to isomorphisms.

\begin{theorem}
\label{theo2}
Let ${\mathbb A}$ and  ${\mathbb A}'$ be two isomorphic non-degenerate evolution algebras. Then $\Delta({\mathbb A})=\Delta({\mathbb A}')$.
\end{theorem}

\begin{proof}

Let us denote by $\phi:{\mathbb A} \to  {\mathbb A}'$ the isomorphism between these algebras and write  by $${\mathbb A}= {\mathbb E}_1 \oplus \ldots  \oplus {\mathbb E}_r$$ the standard extending subspaces decomposition ${\mathbb A}$. By Theorem \ref{princi} we have  $${\mathbb A}'= \phi({\mathbb E}_1) \oplus \ldots  \oplus \phi({\mathbb E}_r).$$ Consider the algebras ${\mathbb A}_i$ and ${\mathbb A}_i'$, where ${\mathbb A}_i': x \cdot_{{\mathbb A}_i'} y =\pi_{{\phi(\mathbb E}_{i})}(x \cdot_{\mathbb A'} y)$. Note that they have the same dimension $\delta_i$. Now, since $x\cdot_{\mathbb A} y = \phi^{-1}(\phi(x)\cdot_{\mathbb A'}\phi(y))$ for $x,y\in {\mathbb E}_i$, we have
 \begin{center}$x\cdot_{{\mathbb A}_i}y = 
\pi_{{\mathbb E}_i}(x\cdot_{\mathbb A}y)= 
\pi_{{\mathbb E}_i}\big(\phi^{-1}(\phi(x)\cdot_{\mathbb A'}\phi(y))\big) = 
\phi^{-1}\big(\pi_{\phi({\mathbb E}_i)}(\phi(x)\cdot_{\mathbb A'}\phi(y))\big) = \phi^{-1}\big(\phi(x)\cdot_{\mathbb A_i'}\phi(y)\big).$ \end{center}
Then, ${\mathbb A}_i^2 = 0$ if and only if $({\mathbb A}_i')^2 = 0$, and they are the algebra ${\bf O}_{\delta_i}$. Moreover, for $x,y,z,w\in {\mathbb E}_i$, we obtain
\begin{longtable}{lcl}
$(x\cdot_{{\mathbb A}_i}y)\cdot_{{\mathbb A}_i}(z\cdot_{{\mathbb A}_i}w)$  &$=$&$ \pi_{{\mathbb E}_i}(\pi_{{\mathbb E}_i}(x\cdot_{\mathbb A}y)\cdot_{\mathbb A}\pi_{{\mathbb E}_i}(z\cdot_{\mathbb A}w))$ \\
&$=$&$\pi_{{\mathbb E}_i}(\phi^{-1}(\phi(\pi_{{\mathbb E}_i}(\phi^{-1}(\phi(x)\cdot_{\mathbb A'}\phi(y))))\cdot_{\mathbb A'}\phi(\pi_{{\mathbb E}_i}(\phi^{-1}(\phi(z)\cdot_{\mathbb A'}\phi(w))))))$\\
&$=$&$\phi^{-1}(\pi_{\phi({\mathbb E}_i)}(\phi(\phi^{-1}(\pi_{\phi({\mathbb E}_i)}(\phi(x)\cdot_{\mathbb A'}\phi(y))))\cdot_{\mathbb A'}\phi(\phi^{-1}(\pi_{\phi({\mathbb E}_i)}(\phi(z)\cdot_{\mathbb A'}\phi(w))))))$ \\
&$=$&$\phi^{-1}(\pi_{\phi({\mathbb E}_i)}(\pi_{\phi({\mathbb E}_i)}(\phi(x)\cdot_{\mathbb A'}\phi(y))\cdot_{\mathbb A'}\pi_{\phi({\mathbb E}_i)}(\phi(z)\cdot_{\mathbb A'}\phi(w))))$\\ 
&$=$&$\phi^{-1}((\phi(x)\cdot_{\mathbb A_i'}\phi(y))\cdot_{\mathbb A_i'}(\phi(z)\cdot_{\mathbb A_i'}\phi(w))).$ 
\end{longtable}
Then, $({\mathbb A}_i^2)^2 = 0$ if and only if 
$(({\mathbb A}_i')^2)^2 = 0$, and they are isomorphic to the algebra ${\bf I}_{\delta_i}$ from Theorem \ref{clasifd1}. Otherwise, they are isomorphic to ${\bf E}_{\delta_i}$.
\end{proof}
From now on, if we give the $\Delta$-trace of an evolution algebra, we are assuming that it is non-degenerate.

\section{The classification method}

In the previous section, we introduced the $\Delta$-trace of a non-degenerate evolution algebra. In this section, we are going to construct a classification method of non-degenerate evolution algebras based on this invariant. 
To achieve this, the following result is essential.

\begin{theorem}
\label{res}
Let ${\mathbb A}$ be a non-degenerate evolution algebra over an algebraically closed field ${\mathbb{F}}$ with $\Delta({\mathbb A}) := (\overline{{\mathbb A}_1},\ldots, \overline{{\mathbb A}_r}).$ Then there exists a natural basis $\bf B$ of ${\mathbb A}$ such that the structure matrix of $\mathbb A$ respect to  $\bf B$ has  the following diagonal blocks:
$${\bf M}_{\bf B}({\mathbb A})=\begin{pmatrix}
\overline{{\mathbb A}_1} & * & * &  * \\
* & \overline{{\mathbb A}_2} &  *&  * \\
* &  * & \ddots &  *\\
* & * & * & \overline{{\mathbb A}_r} \\
\end{pmatrix}.$$

Also, $ {\bf M}_{\bf B}({\mathbb A}) $ has the columns form
$${\bf M}_{\bf B}({\mathbb A})=(v_1\ldots  v_1 | v_2 \ldots  v_2 | \ldots  \ldots  | v_r \ldots  v_r).$$
\end{theorem}
\begin{proof}

Consider  ${\mathbb A}=   {\mathbb E}_1 \oplus \ldots  \oplus {\mathbb E}_r,$
  the extending evolution subspaces decomposition of ${\mathbb A}$.
By Theorem \ref{renorm},  for any $i \in \{1,\ldots,r\}$  there is a natural basis ${\bf B}'_i=\{e_{i,1},\ldots, e_{i,k_i}\}$ of ${\mathbb E}_i$ such that the structure matrix of ${\mathbb A}$ respect to the basis ${\bf B'}:=\cup_{i=1}^{r}{\bf B}'_i$ has the columns form
$${\bf M}_{\bf B'}({\mathbb A})=(v_1\ldots  v_1 | v_2 \ldots  v_2 | \ldots  \ldots  | v_r \ldots  v_r)$$
with  ${\rm dim}({\mathbb{F}}v_i+ {\mathbb{F}}v_j) = 2$ when  $i\neq j$.

Since, for any $i \in \{1,\ldots,r\}$, the algebra ${\mathbb A}_i$ associated to the linear subspace ${\mathbb E}_i$ is isomorphic  to  an algebra
${\bf C}_i \in \{  {\bf O}_{k_i}, {\bf E}_{k_i}, {\bf I}_{k_i}\}$, there exists
a change of natural basis $\phi_i$ from ${\mathbb A}_i$ to ${\bf C}_i$ as given in Theorem \ref{clasifd1}) when  ${\bf C}_i \in
\{{\bf E}_{k_i}, {\bf I}_{k_i}\}$. In case
${\bf C}_i = {\bf O}_{k_i}$ we define $\phi_i={\bf{1}}_{k_i}$. From here, we get for any $i \in \{1,\ldots,r\}$, a   basis ${\bf B}_i=\{\phi(e_{i,1}), \ldots, \phi(e_{i,k_i})\}$ of ${\bf E}_i$.

Consider now the basis ${\bf B}$ of ${\mathbb A}$ given by
${\bf B}=\cup_{i=1}^{r}{\bf B}_i.$
This is a natural basis of ${\mathbb A}$ in satisfying that the  structure matrix  of ${\mathbb A}$ with respect to it is

$${\bf M}_{\bf B}({\mathbb A})=\begin{pmatrix}
\overline{{\mathbb A}_1} & * & * &  * \\
* & \overline{{\mathbb A}_2} &  *&  * \\
* &  * & \ddots &  *\\
* & * & * & \overline{{\mathbb A}_r} \\
\end{pmatrix}.$$

Also, we get as in Theorem \ref{teo100} that $ {\bf M}_{\bf B}({\mathbb A}) $ has the columns form
$${\bf M}_{\bf B}({\mathbb A})=(v_1\ldots  v_1 | v_2 \ldots  v_2 | \ldots  \ldots  | v_r \ldots  v_r).$$
\end{proof}



At this point, fixed a basis ${\bf B}$, to obtain a classification of the non-degenerate evolution algebras of a given dimension $n$ it is enough to classify the evolution algebras ${\mathbb A}$ with structure matrices  as in Theorem \ref{res}.

\begin{definition}
 We denote by  ${\mathcal E}(\overline{{\mathbb A}_1}, \ldots,\overline{ {\mathbb A}_r})$  the set of non-degenerate evolution algebras ${\mathbb A}$ satisfying
 $\Delta({\mathbb A}) = (\overline{{\mathbb A}_1}, \ldots, \overline{{\mathbb A}_r})$.
\end{definition}

\begin{remark}
\label{groups}

The following assertions hold:
\begin{itemize}
\item[(i)] Any isomorphism between two algebras ${\mathbb A}$ and ${\mathbb A}'$ in ${\mathcal E}(\overline{{\mathbb A}_1}, \ldots, \overline{{\mathbb A}_r})$ is of the form  in Remark \ref{baseshape} in such a way  that ${\bf D}_i$ is an automorphism of $\overline{{\mathbb A}_i}$ with ${\bf D}_i^{\rm T} {\bf D}_{i} = \lambda_i  {\bf{1}}_{\delta_i}$, and  ${\bf P}$ permutes blocks $i$ and $j$ of the diagonal if and only if $\overline{{\mathbb A}_i} = \overline{{\mathbb A}_j}$.

 Observe that  this isomorphism  does not depend on the fixed algebras ${\mathbb A}$ and ${\mathbb A}'$ in ${\mathcal E}(\overline{{\mathbb A}_1}, \ldots, \overline{{\mathbb A}_r})$.

\item[(ii)] The set of all these automorphisms  is  a group with the composition denoted by $${\mathcal G}(\overline{{\mathbb A}_1}, \ldots, \overline{{\mathbb A}_r}).$$
This group has the subgroups $${\mathcal G}_{\bf D}(\overline{{\mathbb A}_1}, \ldots, \overline{{\mathbb A}_r})$$ of all the block diagonal (by means of ${\bf D}_{i}$), matrices,  and $${\mathcal G}_{\bf P}(\overline{{\mathbb A}_1}, \ldots, \overline{{\mathbb A}_r})$$ of all the permutations of blocks $i$ and $j$ when $\overline{{\mathbb A}_i}= \overline{{\mathbb A}_j}$.

Any element in ${\mathcal G}(\overline{{\mathbb A}_1}, \ldots, \overline{{\mathbb A}_r})$ is the composition of one element in ${\mathcal G}_{\bf D}(\overline{{\mathbb A}_1}, \ldots, \overline{{\mathbb A}_r})$ and one element in ${\mathcal G}_{\bf P}(\overline{{\mathbb A}_1}, \ldots, \overline{{\mathbb A}_r})$.

\item[(iii)] We have the action

$${\mathcal G}(\overline{{\mathbb A}_1}, \ldots, \overline{{\mathbb A}_r}) \times {\mathcal E}(\overline{{\mathbb A}_1}, \ldots, \overline{{\mathbb A}_r}) \to {\mathcal E}(\overline{{\mathbb A}_1}, \ldots, \overline{{\mathbb A}_r})$$
$$ (\phi, {\bf M}_{\bf B}({\mathbb A})) \mapsto {\phi}^{-1}{\bf M}_{\bf B}({\mathbb A}){\phi}^{(2)}.$$
\end{itemize}

\end{remark}

At this point, we can build a method to classify the non-degenerate evolution algebras. However, this method will require the group of automorphisms of the algebras  ${\bf E_n}$ and ${\bf I}_n$.

\begin{theorem}
\label{th20}
The automorphisms group of ${\bf E}_n$ consists of the
elements $\phi$ of the form:
$$\phi=\begin{pmatrix}
1 &0 & 0 &\ldots   & 0 \\
0 & v_{22} & v_{23} &\ldots  & v_{2 n}  \\
0 & v_{32} & v_{33} &\ldots  & v_{3 n}  \\
\vdots & \vdots &\vdots &  \ddots &\vdots \\
0 & v_{n 2} & v_{n 3} & \ldots  & v_{nn} \\
\end{pmatrix},$$

where $\sum_{k=2}^n v_{ki}^2=1$ and $\sum_{k=2}^n v_{ki} v_{kj}=0$ for $i\neq j$. That is $\phi^{\rm T}\phi={\bf{1}}_{n}$.
\end{theorem}

\begin{proof}
Let $\phi\in {\rm Aut}  ({\bf E}_n)$, the group of automorphisms of ${\bf E}_n$, and denote $\phi(e_i)=v_{1i} e_1+ \ldots + v_{ni}e_n$, then
$$\phi(e_i)\phi(e_i)=(v_{1i}^2+\ldots + v_{ni}^2)e_1.$$
On the one hand, for $i=1$: $$v_{11} e_1+ \ldots + v_{n1} e_n=\phi(e_1)=\phi(e_1^2)=\phi(e_1)\phi(e_1)=(v_{11}^2+\ldots + v_{n1}^2)e_1.$$
Therefore $v_{i1}=0$ for $i>1$ and $v_{11}^2=v_{11}$, then $\phi(e_1)=e_1$.

On the other hand, $\phi(e_i)\phi(e_i)=\phi(e_i^2)=\phi(e_1)=e_1$ implies $(v_{1i}^2+ \ldots + v_{ni}^2)e_1=e_1$, hence
$$(v_{1i}^2+ \ldots + v_{ni}^2)=1.$$
Finally, since $\phi(e_1)\phi(e_i)=e_1 \phi(e_i)=0$, then $v_{1i}=0$.
The converse is a straightforward verification.
\end{proof}

\begin{theorem}
\label{th21}
The automorphisms group of ${\bf I}_n$ consists of the elements of the form:

$$\phi = \begin{pmatrix}
v_{11} & v_{12} & v_{13} &\ldots   & v_{1n} \\
{\bf i}v_{11}- {\bf i} & {\bf i}v_{12}+1 & v_{23} & \ldots  & v_{2n}  \\
v_{31} & {\bf i} v_{31} & v_{33} &  \ldots  & v_{3n}  \\
\vdots & \vdots & \vdots &  \ddots &  \vdots    \\
v_{n1} & {\bf i} v_{n1} & v_{n3} & \ldots  &v_{nn}\\
\end{pmatrix},$$
where $v_{1i}^2+\ldots +v_{ni}^2=v_{11}+ {\bf i}v_{12}$, and 
$\sum_{k=1}^nv_{ki} v_{kj}=0$ for $i\neq j$. That is $\phi^{\rm T}\phi=(v_{11}+ {\bf i}v_{12}) \, {\bf{1}}_{n}$.
\end{theorem}

\begin{proof}
Suppose $\phi\in {\rm Aut}  ({\bf I}_n)$ and denote $\phi(e_i)=v_{1i} e_1+ \ldots + v_{ni} e_n$, then
$$\phi(e_i)\phi(e_i)=(v_{1i}^2+\ldots + v_{ni}^2)(e_1+{\bf i}e_2).$$
We have $\phi(e_i)\phi(e_i)=\phi(e_i^2)=\phi(e_1+{\bf i} e_2)=\phi(e_1)+{\bf i} \phi(e_2)$, hence
\begin{align*}
\label{rel2}
    v_{k2}={\bf i} v_{k1} \textrm{, for $k\geq 3$}.
\end{align*}
Also, $(v_{1i}^2+\ldots + v_{ni}^2)(e_1+{\bf i}e_2)=(v_{11}+{\bf i} v_{12}) e_1+ (v_{21} +{\bf i} v_{22}) e_2$. Thus, 
\begin{center}
${\bf i} v_{11}- v_{12}=v_{21} +{\bf i} v_{22}$ and $v_{1i}^2+\ldots + v_{ni}^2=v_{11}+{\bf i} v_{12}$.
\end{center}

For $i=1,2$, we obtain the following system of equations:
\begin{longtable}{rcl}
 $v_{11}^2+  v_{21}^2+  v_{31}^2+\ldots + v_{n1}^2$&$=$&$v_{11}+{\bf i} v_{12},$\\
$ v_{12}^2+v_{22}^2-v_{31}^2-\ldots - v_{n1}^2$&$=$&$v_{11}+{\bf i} v_{12},$\\
${\bf i} v_{11}- v_{12}$&$=$&$v_{21} +{\bf i} v_{22}.$
\end{longtable}
From where $v_{21}={\bf i} v_{11}-{\bf i}$ and $v_{22}={\bf i}v_{12}+1$.
The converse is a straightforward verification.
\end{proof}

In order to describe the  classes of non-degenerate evolution algebras, we are going to use the notation introduced in \cite{EL2}, since it allows us to study the action
$${\mathcal G}_{\bf P}(\overline{{\mathbb A}_1}, \ldots, \overline{{\mathbb A}_r}) \times {\mathcal E}(\overline{{\mathbb A}_1}, \ldots, \overline{{\mathbb A}_r}) \to {\mathcal E}(\overline{{\mathbb A}_1}, \ldots, \overline{{\mathbb A}_r})$$  in a comfortable way.

Following \cite{EL2}, we recall that given an $n$-dimensional evolution algebra ${\mathbb A}$ with structure matrix ${\bf M}_{\bf B}({\mathbb A}):= (w_{k i})$ respect to the natural basis ${\bf B}$,  the graph $(V,E)$ with $V=\{1,2,\ldots,n\}$ and $E=\{(i,j)\in V \times V: w_{i,j}\neq 0\}$ is called {\it the graph associated to ${\mathbb A}$ respect to {\bf B}}. If we label the graph $(V,E)$ with the map $\omega: E \to {\mathbb{F}}$ given by $\omega((i,j))=w_{i,j}$ we get the so called {\it the weighted  graph associated to ${\mathbb A}$ respect to {\bf B}}.

Observe that given a permutation ${\bf P}\in {\mathcal G}_{\bf P}(\overline{{\mathbb A}_1}, \ldots, \overline{{\mathbb A}_r})$, we have that ${\bf P}^{-1} = {\bf P}^{\rm T}$ and ${\bf P}^{(2)} = {\bf P}$. Therefore, the group ${\mathcal G}_{\bf P}(\overline{{\mathbb A}_1}, \ldots, \overline{{\mathbb A}_r})$ acts on ${\mathcal E}(\overline{{\mathbb A}_1}, \ldots, \overline{{\mathbb A}_r})$ by conjugation. Recall that the class of a graph modulo isomorphisms contains all permutations of the vertices and it can be seen as a representative graph removing the vertices labels. Note that graph isomorphisms act on the adjacency matrix by conjugation. So, it is enough to consider  the graph associated to ${\mathbb A}$ respect to  ${\bf B}$ without labeled vertices, to get a representation of the algebra  ${\mathbb A}$  module the action of ${\mathcal G}_{\bf P}(\overline{{\mathbb A}_1}, \ldots, \overline{{\mathbb A}_r})$.

\bigskip

We have  obtained a method to  classify   $n$-dimensional non-degenerate evolution algebras ${\mathbb A}$. This consists in  following the next  steps:

\begin{enumerate}

\item Fix the possible values of  $\Delta({\mathbb A})$ (expressions  of $n$ as sum of natural numbers).

    \item For any value   ${\mathfrak{s}} \in \Delta({\mathbb A})$, construct the set ${\mathcal E}({\mathfrak{s}})$ which elements are all of the possible  $\Delta({\mathbb A})$ such that
    $\Delta({\mathbb A})={\mathfrak{s}}.$

    \item For any   $\Delta({\mathbb A}) \in {\mathcal E}({\mathfrak{s}})$.

    \begin{enumerate}
        \item Compute the group ${\mathcal G}_{\bf D}(\Delta({\mathbb A}))$.

    \item Calculate the orbits of the action $${\mathcal G}_{\bf D}(\Delta({\mathbb A})) \times  {\mathcal E}(\Delta({\mathbb A})) \to  {\mathcal E}(\Delta({\mathbb A})),$$
(see Remark \ref{groups}-(c)).

    \item (Optional) For every orbit, construct the weighted graph with an adjacency matrix the structure matrix corresponding to one representative of the family. Remove the vertices labels from the graph. In case  we obtain multiple families of orbits for the same $\Delta({\mathbb A})$, there might be a family that contains others for certain values of the weights. Combine them to reduce the number of families.

    \end{enumerate}
\end{enumerate}








\section{Classification of  2-dimensional non-degenerate evolution algebras}

As an application of the method developed in the previous section, we classify  the 2-dimensional non-degenerate evolution algebras in this section.
Fix a basis ${\bf B} = \left\{e_1, e_2\right\}$. For the sake of simplicity, we may refer to an algebra $\mathbb A$ by simply writing its structure matrix in the  basis ${\bf B}$. Also, we may assume that the characteristic of the ground field is zero in the upcoming technical sections.  Now, in order to classify the $2$-dimensional non-degenerate evolution algebras, we have to study the following cases.

\subsection{Case \texorpdfstring{$\delta({\mathbb A})=(1, 1)$}{Lg}}
Suppose the standard extending evolution subspaces decomposition is given by ${\mathbb E}_{1}=\langle e_1 \rangle$, ${\mathbb E}_{2}=\langle e_2 \rangle$. We have 3 subcases.

\subsubsection{Subcase $\Delta({\mathbb A})=({\bf O}_1, {\bf O}_1)$}

In this subcase, we are considering the algebras of the family:

$${\mathcal E}({\bf O}_1, {\bf O}_1)=\left\{{\mathbb A}:{\bf M}_{\bf B}({\mathbb A}) = \begin{pmatrix}
0      & a_{12}\\
a_{21} & 0     \\
\end{pmatrix}, \textrm{ where $a_{12} a_{21} \neq 0$}\right\}.$$
The group ${\mathcal G}_{\bf D}({\bf O}_1, {\bf O}_1)$ consists of maps $\phi$ such that:
    $$\phi = \begin{pmatrix}
    x_1   & 0 \\
    0   & x_2 \\
    \end{pmatrix} \textrm{ where } x_1x_2\neq0$$
and the action of this group on an arbitrary element of ${\mathcal E}({\bf O}_1, {\bf O}_1)$ is given by:
$${\phi}^{-1}{\bf M}_{\bf B}({\mathbb A}){\phi}^{(2)}= \begin{pmatrix}
    0   & a_{12}x_2^2{x^{-1}_1} \\
     a_{21}x_1^2{x_2^{-1}}   & 0 \\
    \end{pmatrix}.$$
By choosing $x_1= (a_{12} a_{21}^{2})^{-\frac{1}{3}}$ and 
$x_2= (a_{12}^{2} a_{21})^{-\frac{1}{3}}$, we obtain a ${\mathcal G}_{\bf D}({\bf O}_1, {\bf O}_1)$-orbit with representative $\begin{pmatrix} 0      & 1 \\
1 & 0    \\
\end{pmatrix}$. Here, the group ${\mathcal G}_{\bf P}({\bf O}_1, {\bf O}_1)=\mathbb S_2$ sends our representative to itself.

\subsubsection{Subcase $\Delta({\mathbb A})=({\bf O}_1, {\bf E}_1)$}

In this subcase, we are considering the algebras of the family:

$${\mathcal E}({\bf O}_1,  {\bf E}_1)=\left\{{\mathbb A}:{\bf M}_{\bf B}({\mathbb A}) = \begin{pmatrix}
0      & a_{12}\\
a_{21} & 1     \\
\end{pmatrix}, \textrm{ where $a_{12} a_{21} \neq 0$}\right\}.$$
The group ${\mathcal G}_{\bf D}({\bf O}_1, {\bf E}_1)$ consists of maps $\phi$ such that:
    $$\phi = \begin{pmatrix}
    x_1   & 0 \\
    0   & 1 \\
    \end{pmatrix} \textrm{ where } x_1\neq0$$
and the action of this group on an arbitrary element of ${\mathcal E}({\bf O}_1, {\bf E}_1)$ is given by:
$${\phi}^{-1}{\bf M}_{\bf B}({\mathbb A}){\phi}^{(2)}= \begin{pmatrix}
    0   & {a_{12}}{x_1^{-1}} \\
    a_{21}x_1^2   & 1 \\
    \end{pmatrix}.$$
If $x_1 = a_{21}^{-\frac{1}{2}}$, we obtain the representatives $\begin{pmatrix}
0      & \alpha \\
1 & 1    \\
\end{pmatrix}$ for $\alpha\in {\mathbb{F}}^{*}$. Here, the group ${\mathcal G}_{\bf P}({\bf O}_1, {\bf E}_1)$ is trivial.

\subsubsection{Subcase $\Delta({\mathbb A})=({\bf E}_1, {\bf E}_1)$}

In this subcase, we are considering the algebras of the family:

$${\mathcal E}({\bf E}_1, {\bf E}_1)=\left\{ {\mathbb A}:{\bf M}_{\bf B}({\mathbb A}) = \begin{pmatrix}
1      & a_{12}\\
a_{21} & 1     \\
\end{pmatrix}, \textrm{ where $a_{12} a_{21} \neq 1$}\right\}.$$
The group ${\mathcal G}_{\bf D}({\bf E}_1, {\bf E}_1)$ is trivial and the action of this group on an arbitrary element of ${\mathcal E}({\bf E}_1, {\bf E}_1)$ sends it to itself. Then, we have one ${\mathcal G}_{\bf D}({\bf E}_1, {\bf E}_1)$-orbit for every element in ${\mathcal E}({\bf E}_1, {\bf E}_1)$.
Now, since the group ${\mathcal G}_{\bf P}({\bf E}_1, {\bf E}_1)=\mathbb S_2$, we obtain the representatives $\begin{pmatrix}
1      & \alpha \\
\beta & 1     \\
\end{pmatrix}$ for $(\alpha, \beta) \in {\mathbb{F}}^2$ with $\alpha \beta \neq 1$, where $(\alpha, \beta)$ and $(\beta, \alpha)$ produce isomorphic algebras.

\subsection{Case \texorpdfstring{$ \delta({\mathbb A})=(2)$}{Lg}}

By Theorem \ref{clasifd1}, ${\mathbb A}$ is either isomorphic to ${\bf E}_2$ or to ${\bf I}_2$.

\bigskip

Summing up this section, we have the following classification theorem.

\begin{theorem}

Given a two-dimensional non-degenerate evolution algebra $\mathbb A$ over an algebraically closed field of characteristic zero ${\mathbb F}$, then it is isomorphic to only one of the following algebras:
\begin{itemize}
    \item If $\delta({\mathbb A})=(1, 1)$. 

    \begin{itemize}
        \item If $\Delta({\mathbb A})=({\bf O}_1, {\bf O}_1)$. Then it is isomorphic to the algebra ${\bf A}_{2,1}:\begin{pmatrix} 0      & 1 \\
1 & 0    \\
\end{pmatrix}$.

        \item If $\Delta({\mathbb A})=({\bf O}_1, {\bf E}_1)$. Then it is isomorphic to ${\bf A}_{2,2}^{\alpha}:\begin{pmatrix}
0      & \alpha \\
1 & 1    \\
\end{pmatrix}$ for some $\alpha\in {\mathbb{F}}^{*}$.

        \item If $\Delta({\mathbb A})=({\bf E}_1, {\bf E}_1)$. Then it is isomorphic to ${\bf A}_{2,3}^{\alpha,\beta}:\begin{pmatrix}
1      & \alpha \\
\beta & 1     \\
\end{pmatrix}$ for $(\alpha, \beta) \in {\mathbb{F}}^2$ and $\alpha \beta \neq 1$.
    \end{itemize}
    The only isomorphisms are between algebras in the same family, and they are permutations of the basis elements. Precisely, the only isomorphisms are ${\mathbb A}_{23}^{\alpha,\beta}\cong {\mathbb A}_{23}^{\beta, \alpha}$.

    \item If $\delta({\mathbb A})=(2)$. Then it is either isomorphic to ${\bf E}_2$ or to ${\bf I}_2$. 

\end{itemize}

\end{theorem}

\section{Classification of  3-dimensional non-degenerate evolution algebras}

By a similar process, we can obtain the classification of the $3$-dimensional non-degenerate evolution algebras.
Fix a basis ${\bf B} = \left\{e_1, e_2, e_3\right\}$. In this section, the following remark will be used.

\begin{remark} By Theorem \ref{th20} and Theorem \ref{th21}, we have the following groups of automorphisms:
\begin{longtable}{lclclcl}

 ${\rm Aut}  ({\bf E}_2)$&$=$&$\left\{\begin{pmatrix}
1   &  0 \\
0 & \pm 1 \\
\end{pmatrix} \right\};$ &  \ &

 ${\rm Aut}  ({\bf I}_2)$&$=$&$\left\{\begin{pmatrix}
x   &  {\bf i} - {\bf i} x\\
{\bf i}x-{\bf i} &  x \\
\end{pmatrix}: x\in {\mathbb{F}}\setminus \left\{\frac{1}{2}\right\} \right\}.$

\end{longtable}
\end{remark}

To classify the $3$-dimensional non-degenerate evolution algebras, we have to study the following cases.

\subsection{Case \texorpdfstring{$\delta({\mathbb A})=(1, 1, 1) $}{Lg}}
Suppose the standard extending evolution subspaces decomposition is given by ${\mathbb E}_{1}=\langle e_1 \rangle$, ${\mathbb E}_{2}=\langle e_2 \rangle$ and ${\mathbb E}_{3}=\langle e_3 \rangle$.

\subsubsection{Subcase $\Delta({\mathbb A})=({\bf O}_1, {\bf O}_1, {\bf O}_1)$} 

In this subcase, we are considering the algebras of the family:

$${\mathcal E}({\bf O}_1, {\bf O}_1, {\bf O}_1)=\left\{{\mathbb A}:{\bf M}_{\bf B}({\mathbb A}) = \begin{pmatrix}
0      & a_{12} & a_{13}\\
a_{21} & 0 & a_{23}    \\
a_{31} & a_{32} & 0    \\
\end{pmatrix} \textrm{ and } \delta({\mathbb A})=(1, 1, 1)\right\}.$$

The group ${\mathcal G}_{\bf D}({\bf O}_1, {\bf O}_1, {\bf O}_1)$ consists of maps $\phi$ such that:
    $$\phi = \begin{pmatrix}
    x_1   & 0 & 0 \\
    0   & x_2 & 0 \\
    0   & 0 & x_3 \\
    \end{pmatrix} \textrm{ where } x_1x_2x_3\neq0$$
and the action of this group on an arbitrary element of ${\mathcal E}({\bf O}_1, {\bf O}_1, {\bf O}_1)$ is given by:
$${\phi}^{-1}{\bf M}_{\bf B}({\mathbb A}){\phi}^{(2)}= \begin{pmatrix}
0      & {a_{12}x_2^2}{x_1^{-1}} & {a_{13}x_3^2}{x_1^{-1}}\\
{a_{21}x_1^2}{x_2^{-1}} & 0 & {a_{23}x_3^2}{x_2^{-1}}    \\
{a_{31}x_1^2}{x_3^{-1}} & {a_{32}x_2^2}{x_3^{-1}} & 0    \\
\end{pmatrix}.$$

Let us denote by $\tau_{ij}$ the transposition that swaps elements $e_i$ and $e_j$ of the basis.

{
\begin{itemize}
    \item If $a_{12}\neq0$, $a_{23}\neq0$ and $a_{31}\neq0$, then we have $\begin{pmatrix}
0      & 1 & \alpha \\
\beta & 0 & 1    \\
1 & \gamma & 0    \\
\end{pmatrix}$, for any $\alpha, \beta, \gamma\in {\mathbb{F}}$.

    \item If $a_{12}=0$, $a_{23}\neq0$ and $a_{31}\neq0$, then we have $\begin{pmatrix}
0      & 0 & \alpha \\
\beta & 0 & 1    \\
1 & 1 & 0    \\
\end{pmatrix}$, for $\alpha, \beta\in {\mathbb{F}}$, $\beta\neq0$. Two cases arise:
\begin{itemize}
    \item If $\alpha\neq0$, choose the permutation $\tau_{12}$ of the basis to obtain an algebra of the first case. 
    \item If $\alpha=0$, then we have a zero row, so it can not be isomorphic to the first case. Therefore, we have a second family $\begin{pmatrix}
0      & 0 & 0 \\
\beta & 0 & 1    \\
1 & 1 & 0    \\
\end{pmatrix}$, for $\beta\in {\mathbb{F}}^*$.
\end{itemize}

    \item If $a_{12}\neq0$, $a_{23}=0$ and $a_{31}\neq0$, then we have $\begin{pmatrix}
0      & 1 & 1\\
\alpha & 0 & 0    \\
1 & \beta & 0    \\
\end{pmatrix}$, for $\alpha, \beta\in {\mathbb{F}}$, $\beta\neq0.$
\begin{itemize}
    \item If $\alpha\neq0$, again choose the permutation $\tau_{12}$ to obtain the first case. 
    \item If $\alpha=0$,  then choose $\tau_{12}$ to obtain the second family. 
\end{itemize}

    \item If $a_{12}\neq0$, $a_{23}\neq0$ and $a_{31}=0$, then we have $\begin{pmatrix}
0      & 1 & \beta \\
1 & 0 & 1    \\
0 & \alpha & 0    \\
\end{pmatrix}$, for $\alpha, \beta\in {\mathbb{F}}$, $\beta\neq0.$  
\begin{itemize}
    \item If $\alpha\neq0$, there is a suitable permutation of the basis to obtain an algebra of the first case. 
    \item If $\alpha=0$, then choose $\tau_{13}$ to obtain the second family. 
\end{itemize}

    \item If $a_{12}=0$, $a_{23}=0$ and $a_{31}\neq0$, then we have $\begin{pmatrix}
0      & 0 & 1\\
\alpha & 0 & 0    \\
1 & 1 & 0    \\
\end{pmatrix}$, for $\alpha\in {\mathbb{F}}^{*}$.

    \item If $a_{12}=0$, $a_{23}\neq0$ and $a_{31}=0$, then we have $\begin{pmatrix}
0      & 0 & \alpha\\
1 & 0 & 1    \\
0 & 1 & 0    \\
\end{pmatrix}$, for $\alpha\in {\mathbb{F}}^{*}$.

    \item If $a_{12}\neq0$, $a_{23}=0$ and $a_{31}=0$, then we have $\begin{pmatrix}
0      & 1 & 1\\
1 & 0 & 0    \\
0 & \alpha & 0    \\
\end{pmatrix}$, for $\alpha\in {\mathbb{F}}^{*}$.

    \item If $a_{12}=0$, $a_{23}=0$ and $a_{31}=0$, then we have $\begin{pmatrix}
0      & 0 & 1\\
1 & 0 & 0    \\
0 & 1 & 0    \\
\end{pmatrix}$. 
\end{itemize}
}

Note that applying $\tau_{12}$ to the last four cases give us the first case.

\subsubsection{Subcase $\Delta({\mathbb A})=({\bf O}_1, {\bf O}_1, {\bf E}_1)$}

In this subcase, we are considering the algebras of the family:

$${\mathcal E}({\bf O}_1, {\bf O}_1, {\bf E}_1)=\left\{{\mathbb A}:{\bf M}_{\bf B}({\mathbb A}) = \begin{pmatrix}
0      & a_{12} & a_{13}\\
a_{21} & 0 & a_{23}    \\
a_{31} & a_{32} & 1    \\
\end{pmatrix} \textrm{ and } \delta({\mathbb A})=(1, 1, 1)\right\}.$$
The group ${\mathcal G}_{\bf D}({\bf O}_1, {\bf O}_1, {\bf E}_1)$ consists of maps $\phi$ such that:
    $$\phi = \begin{pmatrix}
    x_1   & 0 & 0 \\
    0   & x_2 & 0 \\
    0   & 0 & 1 \\
    \end{pmatrix} \textrm{ where } x_1x_2\neq0$$
and the action of this group on an arbitrary element of ${\mathcal E}({\bf O}_1, {\bf O}_1, {\bf E}_1)$ is given by:
$${\phi}^{-1}{\bf M}_{\bf B}({\mathbb A}){\phi}^{(2)}= \begin{pmatrix}
0      & {a_{12}x_2^2}{x_1^{-1}} & {a_{13}}{x_1^{-1}}\\
{a_{21}x_1^2}{x_2^{-1}} & 0 & {a_{23}}{x_2^{-1}}    \\
{a_{31}x_1^2} & {a_{32}x_2^2} & 1    \\
\end{pmatrix}.$$

At this point, the following cases arise:
\begin{itemize}
    \item If $a_{31}\neq0$ and $a_{32}\neq0$, then we have $\begin{pmatrix}
0      & \alpha & \beta\\
\gamma & 0 & \epsilon    \\
1 & 1 & 1    \\
\end{pmatrix}$ for $\alpha, \beta, \gamma, \epsilon\in {\mathbb{F}}$ .

    \item If $a_{31} = 0$ and $a_{32}\neq0$, then $a_{21} \neq 0$, and we have $\begin{pmatrix}
0      & \alpha & \beta\\
1 & 0 & \gamma    \\
0 & 1 & 1    \\
\end{pmatrix}$ for $\alpha, \beta, \gamma \in {\mathbb{F}}$ .

    \item If $a_{31} \neq 0$ and $a_{32} = 0$, then $a_{12} \neq 0$, and choosing $\tau_{12}$, we are in the previous case.

    \item If $a_{31} = 0$ and $a_{32} = 0$, then $a_{12} \neq 0$ and  $a_{21} \neq 0$, and we have $\begin{pmatrix}
0      & 1 & \alpha\\
1 & 0 & \beta    \\
0 & 0 & 1    \\
\end{pmatrix}$ for $\alpha, \beta\in {\mathbb{F}}$ .
\end{itemize}

Recall that the parameters on the previous cases are subject to the condition  $\Delta({\mathbb A})=({\bf O}_1, {\bf E}_1, {\bf E}_1)$.

\subsubsection{Subcase $\Delta({\mathbb A})=({\bf O}_1, {\bf E}_1, {\bf E}_1)$}

In this subcase, we are considering the algebras of the family:

$${\mathcal E}({\bf O}_1, {\bf E}_1, {\bf E}_1)=\left\{{\mathbb A}:{\bf M}_{\bf B}({\mathbb A}) = \begin{pmatrix}
0      & a_{12} & a_{13}\\
a_{21} & 1 & a_{23}    \\
a_{31} & a_{32} & 1    \\
\end{pmatrix} \textrm{ and } \delta({\mathbb A})=(1, 1, 1)\right\}.$$

The group ${\mathcal G}_{\bf D}({\bf O}_1, {\bf E}_1, {\bf E}_1)$ consists of maps $\phi$ such that:
    $$\phi = \begin{pmatrix}
    x_1   & 0 & 0 \\
    0   & 1 & 0 \\
    0   & 0 & 1 \\
    \end{pmatrix}, \textrm{ where } x_1\neq0$$
and the action of this group on an arbitrary element of ${\mathcal E}({\bf O}_1, {\bf E}_1, {\bf E}_1)$ is given by:
$${\phi}^{-1}{\bf M}_{\bf B}({\mathbb A}){\phi}^{(2)}= \begin{pmatrix}
0      & {a_{12}}{x_1^{-1}} & {a_{13}}{x_1^{-1}}\\
{a_{21}x_1^2} & 1 & a_{23}    \\
{a_{31}x_1^2} & a_{32} & 1    \\
\end{pmatrix}.$$

Note that here we have to distinguish the following two cases:
\begin{itemize}
    \item  If $a_{21}\neq0$, then we have $\begin{pmatrix}
0      & \alpha & \beta \\
1 & 1 & \gamma    \\
\epsilon & \zeta & 1    \\
\end{pmatrix}$ for $\alpha, \beta, \gamma,  \epsilon, \zeta \in {\mathbb{F}}$, such that $\Delta({\mathbb A})=({\bf O}_1, {\bf E}_1, {\bf E}_1)$.

    \item If $a_{21}=0$, then $a_{31}\neq0$ and using a suitable permutation we are in the previous case.
\end{itemize}

\subsubsection{Subcase $\Delta({\mathbb A})=({\bf E}_1, {\bf E}_1, {\bf E}_1)$}

In this subcase, we are considering the algebras of the family:

$${\mathcal E}({\bf E}_1, {\bf E}_1, {\bf E}_1)=\left\{{\mathbb A}:{\bf M}_{\bf B}({\mathbb A}) = \begin{pmatrix}
1      & a_{12} & a_{13}\\
a_{21} & 1 & a_{23}    \\
a_{31} & a_{32} & 1    \\
\end{pmatrix} \textrm{ and } \delta({\mathbb A})=(1, 1, 1)\right\}.$$

The group ${\mathcal G}_{\bf D}({\bf E}_1, {\bf E}_1, {\bf E}_1)$ is trivial and the action of this group on an arbitrary element of ${\mathcal E}({\bf E}_1, {\bf E}_1, {\bf E}_1)$ sends it to itself. Then, we have one ${\mathcal G}_{\bf D}({\bf E}_1, {\bf E}_1, {\bf E}_1)$-orbit for every element in ${\mathcal E}({\bf E}_1, {\bf E}_1, {\bf E}_1)$.

\subsection{Case \texorpdfstring{$\delta({\mathbb A})=(2, 1) $}{Lg}}
Suppose the standard extending evolution subspaces decomposition is given by ${\mathbb E}_{1}=\langle e_1, e_2 \rangle$, ${\mathbb E}_{2}=\langle e_3 \rangle$.
Observe that since ${\mathcal G}_{\bf P}$ is trivial, then ${\mathcal G} = {\mathcal G}_{\bf D}$.

\subsubsection{Subcase $\Delta({\mathbb A})=({\bf O}_2, {\bf O}_1)$}

In this subcase, we are considering the algebras of the family:

$${\mathcal E}({\bf O}_2, {\bf O}_1)=\left\{{\mathbb A}:{\bf M}_{\bf B}({\mathbb A}) = \begin{pmatrix}
0       & 0         & a_{13}    \\
0       & 0         & a_{23}    \\
a_{31}  & a_{31}    & 0         \\
\end{pmatrix} \textrm{ and } \delta({\mathbb A})=(2, 1)\right\}.$$

The group ${\mathcal G}({\bf O}_2, {\bf O}_1)$ consists of maps $\phi$ such that:
    $$\phi = \begin{pmatrix}
    x_{11}  & x_{12} & 0 \\
    x_{21}  & x_{22} & 0 \\
    0   & 0 & x_{33} \\
    \end{pmatrix},$$
  where $x_{11}x_{22}x_{33}\neq x_{12}x_{21}x_{33}$  and 
  $\phi_1^{\rm T}\phi_1 = \lambda {\bf{1}}_{2}$  for 
  $\lambda\in {\mathbb{F}}$ and  $\phi_1=\begin{pmatrix}
    x_{11}  & x_{12}  \\
    x_{21}  & x_{22}  \\
    \end{pmatrix}.$ 
    
The action of this group on an arbitrary element of ${\mathcal E}({\bf O}_2, {\bf O}_1)$ is given by:
$${\phi}^{-1}{\bf M}_{\bf B}({\mathbb A}){\phi}^{(2)}= \begin{pmatrix}
0       & 0         & \frac{(a_{23}x_{12}-a_{13}x_{22})x_{33}^2}{x_{12}x_{21}-x_{11}x_{22}}    \\
0       & 0         & \frac{(a_{13}x_{21}-a_{23}x_{11})x_{33}^2}{x_{12}x_{21}-x_{11}x_{22}}    \\
\frac{a_{31}(x_{11}^2+x_{21}^2)}{x_{33}}  & \frac{a_{31}(x_{12}^2+x_{22}^2)}{x_{33}}    & 0         \\
\end{pmatrix},$$ 
where  $x_{11}^2+x_{21}^2=x_{12}^2+x_{22}^2=\lambda.$

Here, we have two cases:

\begin{itemize}
    \item  If $a_{23}^2 + a_{13}^2 \neq 0$, we obtain the algebra corresponding to the matrix  
    $\begin{pmatrix}
0       & 0         & 0    \\
0       & 0         & 1    \\
1  & 1    & 0         \\
\end{pmatrix}$.

    \item If $a_{23}^2 + a_{13}^2 = 0$, we obtain 
    $\begin{pmatrix}
0       & 0         & 1    \\
0       & 0         & i    \\
1  & 1    & 0         \\
\end{pmatrix}$.
    
\end{itemize}

\subsubsection{Subcase $\Delta({\mathbb A})=({\bf O}_2, {\bf E}_1)$}

In this subcase, we are considering the algebras of the family:

$${\mathcal E}({\bf O}_2, {\bf E}_1)=\left\{{\mathbb A}:{\bf M}_{\bf B}({\mathbb A}) = \begin{pmatrix}
0       & 0         & a_{13}    \\
0       & 0         & a_{23}    \\
a_{31}  & a_{31}    & 1         \\
\end{pmatrix} \textrm{ and } \delta({\mathbb A})=(2, 1)\right\}.$$

The group ${\mathcal G}({\bf O}_2, {\bf E}_1)$ consists of maps $\phi$ such that:
    $$\phi = \begin{pmatrix}
    x_{11}  & x_{12} & 0 \\
    x_{21}  & x_{22} & 0 \\
    0   & 0 & 1 \\
    \end{pmatrix},$$ 
    where  $x_{11}x_{22}\neq x_{12}x_{21},$ \ 
    $\phi_1^{\rm T}\phi_1 = \lambda {\bf{1}}_{2}$ for 
    $\lambda\in {\mathbb{F}}$  and $\phi_1=\begin{pmatrix}
    x_{11}  & x_{12}  \\
    x_{21}  & x_{22}  \\
    \end{pmatrix}.$
The action of this group on an arbitrary element of ${\mathcal E}({\bf O}_2, {\bf E}_1)$ is given by:
$${\phi}^{-1}{\bf M}_{\bf B}({\mathbb A}){\phi}^{(2)}= \begin{pmatrix}
0       & 0         & \frac{a_{23}x_{12}-a_{13}x_{22}}{x_{12}x_{21}-x_{11}x_{22}}    \\
0       & 0         & \frac{a_{13}x_{21}-a_{23}x_{11}}{x_{12}x_{21}-x_{11}x_{22}}    \\
a_{31}(x_{11}^2+x_{21}^2)  & a_{31}(x_{12}^2+x_{22}^2)   & 1         \\
\end{pmatrix},$$ 
where $x_{11}^2+x_{21}^2=x_{12}^2+x_{22}^2=\lambda.$

Now, we have to distinguish two cases:

\begin{itemize}
    \item  If $a_{23}^2 + a_{13}^2 \neq 0$, we obtain the family 
    $\begin{pmatrix}
0       & 0         & 0    \\
0       & 0         & 1    \\
\alpha  & \alpha    & 1         \\
\end{pmatrix}$ for $\alpha\in {\mathbb{F}}^*$.

    \item If $a_{23}^2 + a_{13}^2 = 0$, we obtain 
    $\begin{pmatrix}
0       & 0         & 1    \\
0       & 0         & {\bf i}    \\
1  & 1    & 1         \\
\end{pmatrix}$.
    
\end{itemize}

\subsubsection{Subcase $\Delta({\mathbb A})=({\bf E}_2, {\bf O}_1)$}

In this subcase, we are considering the algebras of the family:

$${\mathcal E}({\bf E}_2, {\bf O}_1)=\left\{{\mathbb A}:{\bf M}_{\bf B}({\mathbb A}) = \begin{pmatrix}
1       & 1         & a_{13}    \\
0       & 0         & a_{23}    \\
a_{31}  & a_{31}    & 0         \\
\end{pmatrix} \textrm{ and } \delta({\mathbb A})=(2, 1)\right\}.$$

The group ${\mathcal G}({\bf E}_2, {\bf O}_1)$ consists of maps $\phi$ such that:
    $$\phi = \begin{pmatrix}
    1  & 0 & 0 \\
    0  & \pm 1 & 0 \\
    0   & 0 & x_1 \\
    \end{pmatrix}, \textrm{ where } x_{1}\neq0$$
and the action of this group on an arbitrary element of ${\mathcal E}({\bf E}_2, {\bf O}_1)$ is given by:
$${\phi}^{-1}{\bf M}_{\bf B}({\mathbb A}){\phi}^{(2)}= \begin{pmatrix}
1       & 1         & a_{13} x_1^2    \\
0       & 0         & \pm a_{23} x_1^2    \\
{a_{31}}{x_1^{-1}}  & {a_{31}}{x_1^{-1}}     & 0         \\
\end{pmatrix}.$$

Here, we have two cases:

\begin{itemize}
    \item  If $a_{31} \neq 0$, we obtain the family 
    $\mathbb{W}_1^{\alpha,\beta}: \begin{pmatrix}
1       & 1         & \alpha    \\
0       & 0         & \beta    \\
1       & 1         & 0         \\
\end{pmatrix}$ for $\alpha, \beta\in {\mathbb{F}}$, where $\mathbb{W}^{\alpha,\beta}\cong \mathbb{W}^{\alpha,-\beta}.$

    \item If $a_{31} = 0$, we obtain 
    $\mathbb{W}_2^{\alpha}:\begin{pmatrix}
1       & 1         & \alpha    \\
0       & 0         & 1    \\
0       & 0    & 0         \\
\end{pmatrix}$ for $\alpha\in\mathbb{F}$, where $\mathbb{W}_2^{\alpha}\cong \mathbb{W}_2^{-\alpha}$.
    
\end{itemize}

\subsubsection{Subcase $\Delta({\mathbb A})=({\bf E}_2, {\bf E}_1)$}

In this subcase, we are considering the algebras of the family:

$${\mathcal E}({\bf E}_2, {\bf E}_1)=\left\{{\mathbb A}:{\bf M}_{\bf B}({\mathbb A}) = \begin{pmatrix}
1       & 1         & a_{13}    \\
0       & 0         & a_{23}    \\
a_{31}  & a_{31}    & 1         \\
\end{pmatrix} \textrm{ and } \delta({\mathbb A})=(2, 1)\right\}.$$

The group ${\mathcal G}({\bf E}_2, {\bf E}_1)$ consists of maps $\phi$ such that:
    $$\phi = \begin{pmatrix}
    1  & 0 & 0 \\
    0  & \pm 1 & 0 \\
    0   & 0 & 1 \\
    \end{pmatrix} $$
and the action of this group on an arbitrary element of ${\mathcal E}({\bf E}_2, {\bf E}_1)$ is given by:
$${\phi}^{-1}{\bf M}_{\bf B}({\mathbb A}){\phi}^{(2)}= \begin{pmatrix}
1       & 1         & a_{13}    \\
0       & 0         & \pm a_{23}    \\
a_{31}  & a_{31}    & 1         \\
\end{pmatrix}.$$

Clearly, we have a single family $\mathbb{W}_3^{\alpha, \beta, \gamma}:\begin{pmatrix}
1       & 1         & \beta    \\
0       & 0         & \gamma    \\
\alpha  & \alpha    & 1         \\
\end{pmatrix}$ for $\alpha,\beta,\gamma \in {\mathbb{F}}$, where $\mathbb{W}_3^{\alpha, \beta, \gamma}\cong \mathbb{W}_3^{\alpha,\beta,-\gamma}.$

\subsubsection{Subcase $\Delta({\mathbb A})=({\bf I}_2, {\bf O}_1)$}

In this subcase, we are considering the algebras of the family:

$${\mathcal E}({\bf I}_2, {\bf O}_1)=\left\{{\mathbb A}:{\bf M}_{\bf B}({\mathbb A}) = \begin{pmatrix}
1       & 1         & a_{13}    \\
{\bf i}       & {\bf i}         & a_{23}    \\
a_{31}  & a_{31}    & 0         \\
\end{pmatrix} \textrm{ and } \delta({\mathbb A})=(2, 1)\right\}.$$

The group ${\mathcal G}({\bf I}_2, {\bf O}_1)$ consists of maps $\phi$ such that:
    $$\phi = \begin{pmatrix}
    x_1   &  {\bf i} - {\bf i} x_1 & 0\\
    {\bf i}x_1-{\bf i} &  x_1  & 0\\
    0   & 0 & x_2 \\
    \end{pmatrix}, \textrm{ where } 2x_{1} \neq x_{2}$$
and the action of this group on an arbitrary element of ${\mathcal E}({\bf I}_2, {\bf O}_1)$ is given by:
$${\phi}^{-1}{\bf M}_{\bf B}({\mathbb A}){\phi}^{(2)}= \begin{pmatrix}
1       & 1         &  \frac{({\bf i}a_{23}(x_1-1)+a_{13}x_1)x_2^2}{2x_1-1}   \\
{\bf i}      & {\bf i}        & \frac{({\bf i}a_{13}( 1-x_1)+a_{23}x_1)x_2^2}{2x_1-1}    \\
\frac{a_{31}(2x_1-1)}{x_2}        &  \frac{a_{31}(2x_1-1)}{x_2}    & 0         \\
\end{pmatrix}.$$

From here, we have the cases:
\begin{itemize}
    \item If $ a_{31} \neq 0 $ and $a_{23}^2 + a_{13}^2 \neq 0$, we obtain the family $\begin{pmatrix}
1       & 1         & \alpha    \\
{\bf i}      & {\bf i}         & 0    \\
1  & 1    & 0         \\
\end{pmatrix}$ for $\alpha\in {\mathbb{F}}^*$.

    \item If $ a_{31} \neq 0 $ and $a_{23}^2 + a_{13}^2 = 0$, we have two algebras corresponding to  $\begin{pmatrix}
1       & 1         & 1    \\
{\bf i}       & {\bf i}         & \pm {\bf i}    \\
1  & 1    & 0         \\
\end{pmatrix}$.

    \item If $ a_{31} = 0 $ and $a_{23}^2 + a_{13}^2 \neq 0$, we have $\begin{pmatrix}
1       & 1         & 1    \\
{\bf i}       & {\bf i}         & 0    \\
0  & 0   & 0         \\
\end{pmatrix}$.

\item If $ a_{31} = 0 $ and $a_{23}^2 + a_{13}^2 = 0$, we obtain the algebra $\begin{pmatrix}
1       & 1         & 1    \\
{\bf i}       & {\bf i}         & -{\bf i}    \\
0  & 0   & 0         \\
\end{pmatrix}$.

\end{itemize}

\subsubsection{Subcase $\Delta({\mathbb A})=({\bf I}_2, {\bf E}_1)$}

In this subcase, we are considering the algebras of the family:

$${\mathcal E}({\bf I}_2, {\bf E}_1)=\left\{{\mathbb A}:{\bf M}_{\bf B}({\mathbb A}) = \begin{pmatrix}
1       & 1         & a_{13}    \\
{\bf i}       & {\bf i}         & a_{23}    \\
a_{31}  & a_{31}    & 1         \\
\end{pmatrix} \textrm{ and } \delta({\mathbb A})=(2, 1)\right\}.$$

The group ${\mathcal G}({\bf I}_2, {\bf E}_1)$ consists of maps $\phi$ such that:
    $$\phi =\begin{pmatrix}
    x_1   &  {\bf i} - {\bf i} x_1 & 0\\
    {\bf i}x_1-{\bf i} &  x_1  & 0\\
    0   & 0 & 1 \\
    \end{pmatrix} \textrm{ where } 2x_{1}\neq 1$$
and the action of this group on an arbitrary element of ${\mathcal E}({\bf I}_2, {\bf E}_1)$ is given by:
$${\phi}^{-1}{\bf M}_{\bf B}({\mathbb A}){\phi}^{(2)}= \begin{pmatrix}
1       & 1         &  \frac{{\bf i}a_{23}(x_1-1)+a_{13}x_1}{2x_1-1}   \\
{\bf i}       & {\bf i}         & \frac{{\bf i}a_{13}(1-x_1)+a_{23}x_1}{2x_1-1}    \\
{a_{31}(2x_1-1)}    &  {a_{31}(2x_1-1)}    & 1         \\
\end{pmatrix}.$$

From here, we have the cases:
\begin{itemize}
    \item If $ a_{31} \neq 0 $, we obtain the family $\begin{pmatrix}
1       & 1         & \alpha    \\
{\bf i}       & {\bf i}         & \beta    \\
1  & 1    & 1         \\
\end{pmatrix}$ for $\alpha, \beta\in {\mathbb{F}}$, where $(\alpha,\beta)\neq(1, {\bf i})$.

    \item If $ a_{31} = 0 $ and $a_{23}^2 + a_{13}^2 \neq 0$, we have the algebras $\begin{pmatrix}
1       & 1         & \alpha    \\
{\bf i}       & {\bf i}         & 0    \\
0  & 0   & 1         \\
\end{pmatrix}$ for $\alpha\in {\mathbb{F}}^{*}$.

\item If $ a_{31} = 0 $ and $a_{23}^2 + a_{13}^2 = 0$, we obtain the algebras 
$\begin{pmatrix}
1       & 1         & 1    \\
{\bf i}       & {\bf i}         & {\bf i}    \\
0  & 0   & 1         \\
\end{pmatrix}$ and $\begin{pmatrix}
1       & 1         & \alpha    \\
{\bf i}       & {\bf i}         & - \alpha {\bf i}    \\
0  & 0   & 1         \\
\end{pmatrix}$ for $\alpha\in {\mathbb{F}}$.

\end{itemize}

\subsection{Case \texorpdfstring{$\delta({\mathbb A})=(3)$}{Lg}}

By Theorem \ref{clasifd1}, ${\mathbb A}$ is either isomorphic to ${\bf E}_3$ or to ${\bf I}_3$.

\bigskip

We conclude this section with the following classification theorem.

\begin{theorem}

Given a three-dimensional non-degenerate evolution algebra $\mathbb A$ over an algebraically closed field of characteristic zero ${\mathbb F}$, then it is isomorphic to only one of the following algebras:
\begin{itemize}
    \item If $\delta({\mathbb A})=(1, 1, 1)$. 

    \begin{itemize}
        \item If $\Delta({\mathbb A})=({\bf O}_1, {\bf O}_1, {\bf O}_1)$. Then it is isomorphic to one of the following algebras:
        \begin{tasks}(2)
            \task ${\bf A}_{3,1}^{\alpha,\beta,\gamma}: \begin{pmatrix}
            0      & 1 & \alpha \\
            \beta & 0 & 1    \\
            1 & \gamma & 0    \\
            \end{pmatrix}$,

            \task ${\bf A}_{3,2}^{\alpha}: \begin{pmatrix}
            0      & 0 & 0 \\
            \alpha & 0 & 1    \\
            1 & 1 & 0    \\
            \end{pmatrix}$.      
        \end{tasks}

        \item If $\Delta({\mathbb A})=({\bf O}_1, {\bf O}_1, {\bf E}_1)$. Then it is isomorphic to one of the following algebras:

        \begin{tasks}(3)
             \task ${\bf A}_{3,3}^{\Lambda}: \begin{pmatrix}
                0      & \alpha & \beta\\
                \gamma & 0 & \epsilon    \\
                1 & 1 & 1    \\
                \end{pmatrix}$,

             \task ${\bf A}_{3,4}^{\alpha,\beta,\gamma}: \begin{pmatrix}
                0      & \alpha & \beta\\
                1 & 0 & \gamma    \\
                0 & 1 & 1    \\
                \end{pmatrix}$,

             \task ${\bf A}_{3,5}^{\alpha,\beta}: \begin{pmatrix}
                0      & 1 & \alpha\\
                1 & 0 & \beta    \\
                0 & 0 & 1    \\
                \end{pmatrix}$.
        \end{tasks}

        \item If $\Delta({\mathbb A})=({\bf O}_1, {\bf E}_1, {\bf E}_1)$. It is isomorphic to 
        ${\bf A}_{3,6}^{\Lambda}:\begin{pmatrix}
        0      & \alpha & \beta \\
        1 & 1 & \gamma    \\
        \epsilon & \zeta & 1    \\
        \end{pmatrix}$.

        \item If $\Delta({\mathbb A})=({\bf E}_1, {\bf E}_1, {\bf E}_1)$. It is isomorphic to 
${\bf A}_{3,7}^{\Lambda}:\begin{pmatrix}
1      & \alpha & \beta\\
\gamma & 1 & \epsilon    \\
\zeta  & \xi  & 1    \\
\end{pmatrix}$.

\end{itemize} 

    The only isomorphisms are between algebras in the same family, and they are permutations of the basis elements.
    
    \item If $\delta({\mathbb A})=(2, 1)$. 

    \begin{itemize}
        \item If $\Delta({\mathbb A})=({\bf O}_2, {\bf O}_1)$. Then it is isomorphic to one of the following algebras:
        \begin{tasks}(2)
            \task ${\bf B}_{3,1}: \begin{pmatrix}
            0       & 0         & 0    \\
            0       & 0         & 1    \\
            1  & 1    & 0         \\
            \end{pmatrix}$, 
    
            \task ${ \bf B}_{3,2}:\begin{pmatrix}
            0       & 0         & 1    \\
            0       & 0         & {\bf i}    \\
            1  & 1    & 0         \\
            \end{pmatrix}$. 
        \end{tasks}

        \item If $\Delta({\mathbb A})=({\bf O}_2, {\bf E}_1)$. Then it is isomorphic to one of the following algebras:

        \begin{tasks}(2)
                \task         ${ \bf B}_{3,3}^{\alpha}: \begin{pmatrix}
                0       & 0         & 0    \\
                0       & 0         & 1    \\
                \alpha  & \alpha    & 1         \\
                \end{pmatrix}$, 
                
                \task ${ \bf B}_{3,4}:\begin{pmatrix}
                0       & 0         & 1    \\
                0       & 0         & {\bf i}    \\
                1  & 1    & 1         \\
                \end{pmatrix}$. 
        \end{tasks}

        \item If $\Delta({\mathbb A})=({\bf E}_2, {\bf O}_1)$. Then it is isomorphic to one of the following algebras:

        \begin{tasks}(2)
            \task ${ \bf B}_{3,5}^{\alpha,\beta}: \begin{pmatrix}
            1       & 1         & \alpha    \\
            0       & 0         & \beta    \\
            1       & 1         & 0         \\
            \end{pmatrix}$,

            \task ${ \bf B}_{3,6}^{\alpha}:\begin{pmatrix}
            1       & 1         & \alpha    \\
            0       & 0         & 1    \\
            0       & 0    & 0         \\
            \end{pmatrix}$. 

        \end{tasks}

        The isomorphisms in this section are ${ \bf B}_{3,5}^{\alpha,\beta}\cong { \bf B}_{3,5}^{\alpha,-\beta}$ and ${ \bf B}_{3,6}^{\alpha}\cong { \bf B}_{3,6}^{-\alpha}$.

        \item If $\Delta({\mathbb A})=({\bf E}_2, {\bf E}_1)$. It is isomorphic to ${ \bf B}_{3,7}^{\alpha, \beta, \gamma}: \begin{pmatrix}
1       & 1         & \beta    \\
0       & 0         & \gamma    \\
\alpha  & \alpha    & 1         \\
\end{pmatrix}$. We have ${ \bf B}_{3,7}^{\alpha,\beta,\gamma}\cong { \bf B}_{3,7}^{\alpha,\beta,-\gamma}$.

        \item If $\Delta({\mathbb A})=({\bf I}_2, {\bf O}_1)$. Then it is isomorphic to one of the following algebras:

        \begin{tasks}(2)
        \task ${\bf B}_{3,8}^{\alpha}: \begin{pmatrix}
1       & 1         & \alpha    \\
{\bf i}       & {\bf i}         & 0    \\
1  & 1    & 0         \\
\end{pmatrix}$,

        \task ${\bf B}_{3,9}^{\pm}:\begin{pmatrix}
1       & 1         & 1    \\
{\bf i}       & {\bf i}         & \pm {\bf i}    \\
1  & 1    & 0         \\
\end{pmatrix}$,

        \task ${\bf B}_{3,10}:\begin{pmatrix}
1       & 1         & 1    \\
{\bf i}       & {\bf i}         & 0    \\
0  & 0   & 0         \\
\end{pmatrix}$,

        \task ${\bf B}_{3,11}:\begin{pmatrix}
1       & 1         & 1    \\
{\bf i}       & {\bf i}         & -{\bf i}    \\
0  & 0   & 0         \\
\end{pmatrix}$.
        \end{tasks}

        \item If $\Delta({\mathbb A})=({\bf I}_2, {\bf E}_1)$. Then it is isomorphic to one of the following algebras:

        \begin{tasks}(2)
        \task ${\bf B}_{3,12}^{\alpha,\beta}: \begin{pmatrix}
        1       & 1         & \alpha    \\
        {\bf i}       & {\bf i}         & \beta    \\
        1  & 1    & 1         \\
        \end{pmatrix}$, 

        \task ${\bf B}_{3,13}^{\alpha\neq0}:\begin{pmatrix}
1       & 1         & \alpha    \\
{\bf i}       & {\bf i}         & 0    \\
0  & 0   & 1         \\
\end{pmatrix}$,

        \task ${\bf B}_{3,14}:\begin{pmatrix}
1       & 1         & 1    \\
{\bf i}       & {\bf i}         & {\bf i}    \\
0  & 0   & 1         \\
\end{pmatrix}$,

        \task ${\bf B}_{3,15}^{\alpha}:\begin{pmatrix}
1       & 1         & \alpha    \\
{\bf i}       & {\bf i}         & - \alpha {\bf i}    \\
0  & 0   & 1         \\
\end{pmatrix}$
        \end{tasks}
    
    \end{itemize}

    \item If $\delta({\mathbb A})=(3)$. Then it is either isomorphic to ${\bf E}_3$ or to ${\bf I}_3$. 

\end{itemize}

The parameters in each family ${\bf A}_{3,*}^{*}$ or ${\bf B}_{3,*}^{*}$ are taken freely in $\mathbb F$, but in such a way that the resulting algebra has the $\delta$-index corresponding to that family.

\end{theorem}

\section{Derivations of non-degenerate evolution algebras}

Given a non-degenerate evolution algebra ${\mathbb A}$ such that $\delta({\mathbb A}) = (n_1,  \ldots, n_r)$ and $\Delta({\mathbb A}) := (\overline{{\mathbb A}_1},\ldots, \overline{{\mathbb A}_r})$. Fix a basis ${\bf B}:= \left\{e_1, \ldots, e_n\right\}$ as in Theorem \ref{res}, then for a derivation ${\mathfrak D}$ with matrix $(d_{ij})$ of ${\mathbb A}$ we have:
$${\mathfrak D}(e_i e_j ) = {\mathfrak D}(e_i)e_j +e_i {\mathfrak D}(e_j).$$
If $i\neq j$, then  $d_{ij}e_i^2 + d_{ji}e_j^2 = 0$. Therefore, $d_{ij}=-d_{ji}$ if $e_i, e_j \in {\mathbb A}_k$ for some $k$ or $d_{ij}=d_{ji} = 0 $ otherwise. From here, the matrix of ${\mathfrak D}$ is a block diagonal matrix with blocks $D_k$ for $1\leq k\leq r$ of size $n_{k}\times n_{k}$, where $D_k$ is a diagonal matrix plus an antisymmetric matrix.
Now, if $i=j$, then ${\mathfrak D}(e_i^2) = 2 e_i {\mathfrak D}(e_i)$ if and only if ${\mathfrak D} (e_i^2) = 2 d_{ii} e_i^2$ (note that $e_i^2$ is an eigenvector of ${\mathfrak D}$ corresponding to the eigenvalue $2 d_{ii}$). This is equivalent, for any $1\leq i,j \leq n$, to the equations:
$$\sum_{k} d_{jk}w_{ki} = 2 d_{ii}w_{ji},$$
where ${\bf M}_{\bf B}({\mathbb A}) := (w_{ki})$. 
Also, note that if $e_i, e_j \in {\mathbb A}_k$, then ${\mathfrak D} (e_i^2) = {\mathfrak D} (e_j^2)$,  $2 d_{ii} e_i^2 = 2 d_{jj} e_{j}^2$ and $d_{ii} = d_{jj}$. Let us assume in this section that the characteristic of the algebraically closed ground field $\mathbb{F}$ is not two.

\begin{theorem}
\label{deri}
    The map corresponding to the matrix $D_k$ is a derivation of $\overline{{\mathbb A}_k}$ for $1\leq k \leq r$.    
\end{theorem}
\begin{proof}

It follows by observing that since ${\mathfrak D}$ is block diagonal, then ${\mathfrak D}(e_i^2) = 2 e_i {\mathfrak D}(e_i)$ can be written 
$$\begin{pmatrix}
D_1 \overline{{\mathbb A}_1} & D_1 {\mathbb A}_{12} & \ldots  &  D_1 {\mathbb A}_{1r} \\
D_2 {\mathbb A}_{21} & D_2\overline{{\mathbb A}_2} &  \ldots  &  D_2{\mathbb A}_{2r} \\
\vdots &  \vdots & \ddots &  \vdots\\
D_r{\mathbb A}_{r1} & D_r{\mathbb A}_{r2} & \ldots  & D_r\overline{{\mathbb A}_r} \\
\end{pmatrix} = 2 \begin{pmatrix}
d_{11}\overline{{\mathbb A}_1} & d_{22}{\mathbb A}_{12} & \ldots  &  d_{rr}{\mathbb A}_{1r} \\
d_{11}{\mathbb A}_{21} & d_{22}\overline{{\mathbb A}_2} &  \ldots  &  d_{rr}{\mathbb A}_{2r} \\
\vdots &  \vdots & \ddots &  \vdots\\
d_{11}{\mathbb A}_{r1} & d_{22}{\mathbb A}_{r2} & \ldots  & d_{rr}\overline{{\mathbb A}_r} \\
\end{pmatrix},$$
where ${\bf M}_{\bf B}({\mathbb A}) = \begin{pmatrix}
\overline{{\mathbb A}_1} &  \ldots  &  {\mathbb A}_{1r} \\
\vdots  & \ddots &  \vdots\\
{\mathbb A}_{r1}  & \ldots  & \overline{{\mathbb A}_r} \\
\end{pmatrix}$. Hence, we have $D_{k} \overline{{\mathbb A}_k} = 2 d_{kk}\overline{{\mathbb A}_k}$, so $D_{k}$ is a derivation of $\overline{{\mathbb A}_k}$.
\end{proof}

    Observe  that if $\mathbb{A}$ is the direct sum of the algebras $\overline{{\mathbb A}_k}$, then the converse of Theorem \ref{deri} is true. However, in general, it is not true, as we need the additional condition of $e_i^2$ being an eigenvector of the derivation corresponding to the eigenvalue $2 d_{ii}$ for every $1 \leq i\leq {\rm dim}(\mathbb{A})$.

Now, by the previous result, we are interested in the derivations of ${\bf E}_n$ and ${\bf I}_n$ (for ${\bf O}_n$ the result is trivial).

\begin{lemma}
\label{deren}
    The derivations ${\mathfrak D}$ of ${\bf E}_n$ are of the form ${\mathfrak D} = \begin{pmatrix}
0 &  0 \\
0 &  D_0 \\
\end{pmatrix}$, where $D_0 = - D_0^{\rm T} = (d_{ij})_{i,j}^{n-1}$. Moreover, $\dim \mathfrak{Der} ({\bf E}_n) = \frac{(n-1)(n-2)}{2}$.
    
    
\end{lemma}
\begin{proof}

Denote the matrix of a derivation ${\mathfrak D}$ as $(d_{ij})$ in a convenient basis. Since, for $i\neq j$, we have ${\mathfrak D}(e_i e_j) = {\mathfrak D}(e_i)e_j + {\mathfrak D}(e_j)e_i = 0$, then $d_{ij}=-d_{ji}$. Also, ${\mathfrak D}(e_i^2) = 2{\mathfrak D}(e_i)e_i$ implies that ${\mathfrak D}(e_1) = 2 d_{ii}e_1$, so $d_{ii}=d_{jj}$. At last, the equation $\sum_{k}d_{k1}e_{k} = 2 d_{11}e_1$ implies that $d_{k1}=0$ for $1 \leq k\leq n$.

\end{proof}

\begin{lemma}
\label{derin}
    The derivations ${\mathfrak D}$ of ${\bf I}_n$ are of the form:
$$D = \begin{pmatrix}
d & - d {\bf i}  & -   d_{23} {\bf i}  &\ldots   & -   d_{2 (n-1)} {\bf i}  & -   d_{2 n} {\bf i}   \\
  d{\bf i}  & d & d_{23} &\ldots  & d_{2 (n-1)} & d_{2 n}  \\
  d_{23} {\bf i}  & -d_{23} & d &\ldots  & d_{3 (n-1)} & d_{3 n}  \\
\vdots & \vdots &\vdots &  \ddots & \vdots &\vdots \\
  d_{2 (n-1)}  {\bf i}  & -d_{2 (n-1)} & -d_{3 (n-1)} & \ldots  &d & d_{(n-1) n} \\
  d_{2 n}  {\bf i}  & -d_{2 n} & -d_{3 n} & \ldots  &-d_{(n-1) n} & d \\
\end{pmatrix}$$
    Moreover, $\dim \mathfrak{Der} ({\bf I}_n) = \frac{(n-1)(n-2)}{2} + 1$.
\end{lemma}
\begin{proof}

Denote the matrix of a derivation ${\mathfrak D}$ as $(d_{ij})$ in a convenient basis. Again, for $i\neq j$, then $d_{ij}=-d_{ji}$. Since ${\mathfrak D}(e_i^2)=2{\mathfrak D}(e_i)e_i$, we have ${\mathfrak D}(e_1 + {\bf i} e_2) = 2 d_{ii}(e_1 + e_2)$. Hence,  
\begin{center}
    $2d_{jj}e_1 + 2 d_{jj} {\bf i}e_2 = 2 d_{ii}e_1 + 2 d_{ii} {\bf i} e_2$, \ so $d:=d_{ii}=d_{jj} $.
\end{center}
\noindent Moreover, we can write 
$$\sum_{k}(d_{k1} +  d_{k2}{\bf i})e_k = 2 d_{ii}e_1 + 2 d_{ii}{\bf i}e_2.$$
For $k=1$, we have $d_{11} +  d_{12} {\bf i} = 2 d_{ii}$, thus, 
$d_{12} = - d{\bf i}$. For $k = 2$, we have $d_{21} +  d_{22} {\bf i} = 2d_{ii}{\bf i}$, so $d_{21}=d{\bf i}$. Finally, for $k>2$, we have  $d_{k1}=-d_{k2}{\bf i}$.
\end{proof}

By the previous results, every derivation of $\bf E_{n}$ or $\bf I_{n}$ is singular. However, there are examples of non-degenerate evolution algebras with non-singular derivations, such as 
\begin{center}
${\mathbb A}:e_1^2 = e_3,\, e_2^2 = e_3, \, e_3^2 = e_1 + {\bf i} e_2$.
\end{center}
Hence, we have the following corollary.

\begin{corollary}
    Given a non-degenerate evolution algebra with a non-singular derivation, then it has $\Delta$-trace equal to $({\bf O}_{n_1}, \ldots, {\bf O}_{n_r})$, for some ${n_1}, \ldots, {n_r}\in {\mathbb N}$.
\end{corollary}

The well-known fact that every derivation of a perfect evolution algebra is zero can be generalized in the following way. This result was first proven in a recent paper \cite{yolder}, in order to enrich the evolution algebras literature we present a different, but interesting, approach to it.

\begin{theorem}
Suppose $\textrm{char}\, \mathbb{F}$ is not a Mersenne prime. Given an evolution algebra ${\mathbb A}$ with $\delta({\mathbb A}) = (1,  \ldots, 1)$, then every derivation is zero. 
\end{theorem}
\begin{proof}

Suppose $\Delta({\mathbb A}) = ({\bf O}_1,  \ldots, {\bf O}_1, {\bf E}_1,  \ldots, {\bf E}_1)$, where we have $n_1$ copies of ${\bf O}_1$ and $n_2$ copies of ${\bf E}_1$ for 
$0 \leq n_1, n_2 \leq n$  such that $n_1 + n_2 = n$. Fix a basis ${\bf B}:= \left\{e_1, \ldots, e_n\right\}$ as in Theorem \ref{res} and denote the structure matrix of  $\mathbb A$ as $(w_{ki})$. Note that for every $1\leq i\leq n$ there is some $k(i)$ such that $w_{k(i) i}\neq 0$. Choose any map $k:\left\{1, \ldots, n\right\}\rightarrow \left\{1, \ldots, n\right\}$ satisfying that $w_{k(i)i}\neq0$ for every $1\leq i\leq n$. It is clear that since $\mathbb A$ is non-degenerate, such a map exists.

Now, by Theorem \ref{deri}, a derivation $\mathfrak D$ of ${\mathbb A}$ has a diagonal matrix with diagonal $(d_1, \ldots, d_{n_1}, 0, \ldots, 0)$ for some $d_i\in {\mathbb{F}}$. Then, the equation $\sum_{k} d_{jk}w_{ki} = 2 d_{ii}w_{ji}$ is equivalent to $d_{j}w_{ji} = 2 d_{i}w_{ji}$ and using the map $k$ we can write $d_{k^{s}(i)} = 2^{s}d_{i}$, where $k^{s}$ denotes applying the map $k$ for $s$ times.
Now, for any $i\leq n_1$ three scenarios are possible:
\begin{enumerate}
    \item If $n_1 < k^{s}(i)$ for some $s$, then $d_{k^{s}(i)} = 0$ and $d_{i}=0$.

    \item If $k^{s}(i) = i$ for some $s$, then $d_{i} = 2^{s} d_{i}$ and we have $d_{i}=0$, since $\textrm{char}\, \mathbb{F}$ is not a Mersenne prime.

    \item If $k^{s}(i)\neq i$ and $k^{s}(i) \leq n_1 $ for any $s$, then there exist some $s\neq 0$ and some $p\neq0$ such that $k^s(i) = k^{s+p}(i)$ and we have that $d_{k^{s}(i)} = 2^{p}d_{k^{s}(i)}$. Therefore, $d_{k^{s}(i)}=0$ and $d_{i}=0$.
\end{enumerate}

Hence, we conclude that $d_1 = \ldots  = d_{n_1} = 0$.
\end{proof}

The previous result is not true if the characteristic is a Mersenne prime.

\begin{example}
    Suppose $p=2^{n}-1$ is a Mersenne prime. Consider the $n$-dimensional evolution algebra given by $e_{i}^{2}=e_{i+1}$ for $i< n$ and $e_{n}^2=e_1$. We show that the homomorphism $\mathfrak D$ defined by $\mathfrak D(e_{i}) = 2^{i-1}e_i$ for $1\leq i\leq n$ is a derivation. If $i\neq j$, then 
$ 0 = {\mathfrak D}(e_{i}e_{j}) = {\mathfrak D}(e_{i})e_{j} + e_{i}{\mathfrak D}(e_{j})=0$. Now, if $i<n$, then $${\mathfrak D}(e_{i}^2) = {\mathfrak D}(e_{i+1})=2^{i}e_{i+1} = 2 (2^{i-1}e_{i}^2)= 2 {\mathfrak D}(e_{i})e_{i}.$$
If $i=n$, then 
$${\mathfrak D}(e_{n}^2) = {\mathfrak D}(e_{1})= e_{1} = e_{n}^2 = 
2^{n}e_n^2 = 2(2^{n-1}e_{n}^2) = 2 {\mathfrak D}(e_{n})e_{n}.$$
So ${\mathfrak D}$ is a derivation.
\end{example}

\section{The variety of evolution algebras with (one or zero)-dimensional square }

Although it is known that the class of evolution algebras is not a variety, the study of some varieties contained in it is still an interesting topic. In this section, we study the class of complex evolution algebras with (one or zero)-dimensional square, and we prove that it is an irreducible variety.

\begin{theorem}
Any complex $n$-dimensional commutative algebra $\mathbb A$ with $\dim {\mathbb A}^2 = 1$ is an evolution algebra. 
\end{theorem}
\begin{proof}
    Let $\mathbb A$ be a complex commutative algebra such that  $\dim {\mathbb A}^2 = 1$. Suppose ${\mathbb A}^2 = \langle e_1 \rangle $, then fix a basis ${\bf B}=\left\{e_1,\ldots, e_n\right\}$ and let $e_ie_j=m_{ij}e_1$ for any $1\leq i, j \leq n$, where $m_{ij}\in {\mathbb C}$.
    Consider its multiplication structure matrices $M_1=m_{ij}$, $M_k=(0)$ for $k\neq 1$. Then $\mathbb A$ is an evolution algebra if and only if $M_1, \ldots, M_n$ (see \cite{velasco}) are simultaneously diagonalizable via congruence. But, we only have one matrix to diagonalize, and it is symmetric, so by Takagi theorem, there exists a unitary matrix $V$ such that $V^TM_1V=D$ for some real diagonal matrix with non-negative entries. Hence, $\mathbb A$ is an evolution algebra.   
\end{proof}


    
    

As a consequence, we have the following result.

\begin{corollary}
    The class of complex $n$-dimensional evolution algebras with one-dimensional squares together with trivial algebra is a variety,  that we will be called $\mathfrak{E^{n}}.$
\end{corollary}

In the following result, we prove that this variety is irreducible. For that, we have to introduce some notions. For a further introduction to the topic of degeneration, see for example \cite{degs22,ikp22} and the references therein. 
Given an $n$-dimensional vector space ${\bf V}$, the set ${\rm Hom}({\bf V} \otimes {\bf V},{\bf V}) \cong {\bf V}^* \otimes {\bf V}^* \otimes {\bf V}$ 
is a vector space of dimension $n^3$. This space has a structure of the affine variety $\mathbb{C}^{n^3}.$ 
Indeed, let us fix a basis $e_1,\dots,e_n$ of ${\bf V}$. Then any $\mu\in {\rm Hom}({\bf V} \otimes {\bf V},{\bf V})$ is determined by $n^3$ structure constants $c_{i,j}^k\in\mathbb{C}$ such that
$\mu(e_i\otimes e_j)=\sum_{k=1}^nc_{i,j}^ke_k$. A subset of ${\rm Hom}({\bf V} \otimes {\bf V},{\bf V})$ is {\it Zariski-closed} if it can be defined by a set of polynomial equations in the variables $c_{i,j}^k$ ($1\le i,j,k\le n$).

Let $T$ be a set of polynomial identities.
The algebra structures on ${\bf V}$ satisfying all the polynomial identities from $T$ form a Zariski-closed subset of the variety ${\rm Hom}({\bf V} \otimes {\bf V},{\bf V})$. We denote this subset by $\mathbb{L}(T)$.
The general linear group ${\rm GL}({\bf V})$ acts on $\mathbb{L}(T)$ by conjugation:
$$ (g * \mu )(x\otimes y) = g\mu(g^{-1}x\otimes g^{-1}y)$$ 
for $x,y\in {\bf V}$, $\mu\in \mathbb{L}(T)\subset {\rm Hom}({\bf V} \otimes {\bf V},{\bf V})$ and $g\in {\rm GL}({\bf V})$.
Thus, $\mathbb{L}(T)$ decomposes into the ${\rm GL}({\bf V})$-orbits that correspond to the isomorphism classes of algebras. 
Let $O(\mu)$ denote the ${\rm GL}({\bf V})$-orbit of $\mu\in\mathbb{L}(T)$ and $\overline{O(\mu)}$ its Zariski closure.
Let ${\bf A}$ and ${\bf B}$ be two $n$-dimensional algebras satisfying the identities from $T$ and $\mu,\lambda \in \mathbb{L}(T)$ represent ${\bf A}$ and ${\bf B}$ respectively.
We say that ${\bf A}$ {\it degenerates to} ${\bf B}$ and write ${\bf A}\to {\bf B}$ if $\lambda\in\overline{O(\mu)}$.
Note that in this case we have $\overline{O(\lambda)}\subset\overline{O(\mu)}$. Hence, the definition of degeneration does not depend on the choice of $\mu$ and $\lambda$. If ${\bf A}\not\cong {\bf B}$, then the assertion ${\bf A}\to {\bf B}$ 
is called a {\it proper degeneration}. We write ${\bf A}\not\to {\bf B}$ if $\lambda\not\in\overline{O(\mu)}$.
Let ${\bf A}$ be represented by $\mu\in\mathbb{L}(T)$. Then  ${\bf A}$ is  {\it rigid} in $\mathbb{L}(T)$ if $O(\mu)$ is an open subset of $\mathbb{L}(T)$.
Recall that a subset of a variety is called {\it irreducible} if it cannot be represented as a union of two non-trivial closed subsets. A maximal irreducible closed subset of a variety is called an {\it irreducible component}.
It is well known that any affine variety can be represented as a finite union of its irreducible components in a unique way.
The algebra ${\bf A}$ is rigid in $\mathbb{L}(T)$ if and only if $\overline{O(\mu)}$ is an irreducible component of $\mathbb{L}(T)$.

The last result characterizes the variety $\mathfrak{E^{n}}.$

 \begin{theorem}\label{Enirreducible}
The variety   $\mathfrak{E^{n}}$ is irreducible
and ${\bf E}_n$ is rigid.  
\end{theorem}

\begin{proof}
The case of $n=1$ is trivial. Hence, we will consider $n>1.$
Thanks to \cite{candido},  all algebras from  $\mathfrak{E^{n}}$
have the following type:
\begin{enumerate}
    \item nilpotent algebras.
    \item non-nilpotent algebras of type ${\bf E}_k \oplus {\mathbb C}^{n-k}$ for $1 \leq k\leq n.$
    \item non-nilpotent algebras of type ${\bf I}_k \oplus {\mathbb C}^{n-k}$ for $2 \leq k\leq n.$

\end{enumerate}

 The non-degenerate algebras  in this variety are precisely ${\bf E}_n$ and  ${\bf I}_n$. In Lemmas \ref{deren} and \ref{derin}, we have shown that the dimension of the algebra of derivations of the two evolution algebras with one-dimensional square and  with trivial annihilator differ in one unit. Namely,   $\dim \mathfrak{Der} ({\bf I}_n) = \dim \mathfrak{Der} ({\bf E}_n) + 1$. Recall that the dimension of the orbit of a $n$-dimensional algebra $\mathbb{A}$ is $n^2 - \dim \mathfrak{Der} (\mathbb{A})$. Hence, the dimensions of the orbits of ${\bf E}_n$ and ${\bf I}_n$ also differ in one unit. Moreover, we have ${\bf E}_n \to {\bf I}_n$, as we show below.

 Consider the parametrized isomorphism given by
    $$g(e_{1})(t)= (1+t^2)e_1 + {\bf i} e_2, \quad g(e_{2})(t)= {\bf i} (1-t^2)e_1 - e_2, \quad  g(e_k)(t)= \sqrt{2}t e_k, \mbox{ \ for \ }3\leq k\leq n,$$
       and its inverse, given by

     $$g(e_{1})(t)= \frac{1}{2t^2} e_1 + \frac{{\bf i}}{2t^2}e_2, 
     \quad g(e_{2})(t)=  \frac{{\bf i}(1-t^2)}{2t^2}e_1 - \frac{1+t^2}{2t^2}e_2, \quad  g(e_k)(t)= \frac{1}{\sqrt{2}t} e_k.$$
    Then, the action of $g$ gives us the family of commutative algebras $g(t)*{\bf E}_n$, given by
    \begin{longtable}{lcllcl}
    $e_1^2 $&$=$&$ (1+\frac{t^2}{2}) e_1 + (1+\frac{t^2}{2}) {\bf i} e_2,$ & $e_1 e_2 $&$=$&$ -\frac{{\bf i}t^2}{2}  e_1 + \frac{t^2}{2} e_2,$\\
    $e_2 ^2$&$ = $&$(1-\frac{t^2}{2})e_1 + (1-\frac{t^2}{2}) {\bf i} e_2,$&$ e_k^2$&$= $&$e_1 + {\bf i} e_2, $
    \end{longtable}
  \noindent  for $3\leq k\leq n$.  
    From here, it is clear that $\lim\limits_{t\to 0}g(t)*{\bf E}_n = {\bf I}_n$. Moreover, it follows  that this degeneration is primary by the dimension of the algebras of derivations.   

Thanks to \cite{ikp22}, the variety of complex $n$-dimensional nilpotent commutative algebras with dimension of the square not greater than one is irreducible 
with generic algebra ${\bf N}_n$ given by the following multiplication table $e_2^2=\ldots =e_{n}^2= e_1.$

Then, the rest of the necessary  degenerations are (here $g$ is inverse to the parametrized isomorphism):
\begin{longtable}{lclclcllcll}
${\bf E}_n$&$ \to $&${\bf N}_n$&$:$
&$g(e_1)$&$=$&$t^2 e_1,$ &$g(e_k)$&$=$&$t e_k,$ &$ \mbox{for} \ 2\leq k \leq n.$\\

${\bf E}_k \oplus {\mathbb C}^{n-k} $&$\to$&$ {\bf E}_{k-1} \oplus {\mathbb C}^{n-k+1}$&$:$&
$g(e_k)$&$=$&$  t e_k,$ &$g(e_m)$&$=$&$ e_m,$ &$\mbox{for} \ 1\leq m\neq   k \leq n.$\\

${\bf I}_k \oplus {\mathbb C}^{n-k}$&$ \to$&$ {\bf I}_{k-1} \oplus {\mathbb C}^{n-k+1}$&$:$&
$g(e_k)$&$= $&$ t e_k,$ &$g(e_m)$&$=$&$ e_m,$ &$\mbox{for} \ 1\leq m\neq   k \leq n.$
\end{longtable}
Summarizing, we have shown that every  algebra in $\mathfrak{E^{n}}$  is in the closure of the orbit of ${\bf E}_n.$
\end{proof}

By Theorem \ref{Enirreducible}, in order to describe the identities  defining the variety  $\mathfrak{E^{n}}$, it is enough to study the polynomial identities satisfied by ${\bf E}_n$. The following result slightly simplifies this problem.

\begin{theorem}
Given a identity of degree $n$. Then if  ${\bf E}_{n}$ satisfies the identity, then  ${\bf E}_{k}$ satisfies the identity too for ${k\in \mathbb N}$.    
\end{theorem}
\begin{proof}

Suppose  ${\bf E}_{n}$ satisfies an identity $p(x_{1},\ldots, x_{n}) = 0 $ of degree $n$. If $k\leq n$, then ${\bf E}_{k}$ satisfies the identity, because ${\bf E}_{n}\to {\bf E}_{k}\oplus {\mathbb C}^{n-k}$, so ${\bf E}_{k}\oplus {\mathbb C}^{n-k}$ satisfies it and then ${\bf E}_{k}\oplus {\mathbb C}^{n-k}/{\mathbb C}^{n-k}\cong{\bf E}_{k}$ satisfies it.

If $k>n$, then we have that $p(e_{i_1}, \ldots, e_{i_n})=0 $ for $i_l\leq n$. Now, for $p(e_{j_1}, \ldots, e_{j_n})$ where $j_l\leq k$, we assign to every index ${j_l}$ an index $i_l\leq n$ using the following rules:
$j_l = 1$ if and only if $i_l = 1$  and $j_l = j_m$ if and only if $i_l =  i_m$. Then we have $p(e_{j_1}, \ldots, e_{j_n})=p(e_{i_1}, \ldots, e_{i_n}) = 0$.    
\end{proof}

As a consequence, if we want to check if ${\bf E}_{n}$ satisfies an identity of degree $k\leq n$, we check it for ${\bf E}_{k}$.

\begin{proposition}
The variety  $\mathfrak{E^{n}}$ does not satisfy nontrivial identities of degree three, 
but it satisfies a nontrivial identity of degree four: $((xy)z)t=((xy)t)z.$ 
\end{proposition}

The present proposition can be obtained by a direct calculation and it gives the following question.

\medskip

\noindent
{\bf Open question}.
Characterize the variety  $\mathfrak{E^{n}}$ with a set of identities.

\end{document}